\documentclass[a4paper,reqno]{article}

\usepackage{latexsym,amsmath,mathtools,epsfig,epsf}
\usepackage{amsmath,amsthm,amstext,amscd,amssymb,euscript,url}

\usepackage{comment}
\usepackage{mathrsfs}
\usepackage{bbm}
\usepackage[english]{babel}
\usepackage[normalem]{ulem}
\usepackage{hyperref}
\usepackage[utf8]{inputenc}
\usepackage{caption}
\usepackage{subcaption}
\usepackage{color}
\usepackage{enumerate}
\usepackage[title]{appendix}

\allowdisplaybreaks

\numberwithin{equation}{section}

\newtheorem{theorem}{Theorem}[section]
\newtheorem{lemma}[theorem]{Lemma}
\newtheorem{corollary}[theorem]{Corollary}
\newtheorem{proposition}[theorem]{Proposition} 

\theoremstyle{remark}
\newtheorem{remark}[theorem]{Remark}

\theoremstyle{definition}
\newtheorem{assumption}{Assumption}
\newtheorem{definition}[theorem]{Definition}

\def\N{{\mathbb N}}

\def\R{{\mathbb R}}

\def\caX{{\mathcal X}}
\newcommand{\E}{{\mathbb E}}
\newcommand{\ERW}{{\mathbb E}^{\text{IRW}}}
\renewcommand{\P}{{\mathbb P}}
\newcommand{\Bin}{\text{\normalfont Bin}}
\newcommand{\Poi}{\text{\normalfont Poi}}
\newcommand{\NegBin}{\text{\normalfont NB}}

\newcommand{\hyp}[5]{\mathstrut_{#1} F_{#2} \left( {\left. \genfrac{}{}{0pt}{} {#3 } { #4 }  \right\vert {#5}} \right)}

\newcommand{\set}[1]{{\left \{ #1 \right \}}}
\newcommand{\abs}[1]{{\left | #1 \right |}}
\newcommand{\one}{\mathbbm{1}}

\newcommand{\bra}[1]{\left( #1 \right) }
\newcommand{\sbra}[1]{\left[ #1 \right] }
\newcommand{\dx}{\:\mathrm{d}}
\newcommand{\dd}{\mathrm{d}}

\newcommand{\sd}{\, \cdot \,}

\makeatother

\begin{document}

\title{Intertwining and Duality for Consistent Markov Processes}

\author{Simone Floreani$^1$, Sabine Jansen$^2$, Frank Redig$^3$,\\ 
Stefan Wagner$^4$}

\date{\today}

\maketitle

\begin{abstract}
In this paper we derive intertwining relations for a broad class of conservative particle systems both in discrete and continuous setting. Using the language of point process theory, we are able to derive a natural framework in which duality and intertwining can be formulated.
We prove falling factorial and orthogonal polynomial intertwining relations in a general setting. These intertwinings unite the previously found ``classical’’ and orthogonal self-dualities in the context of discrete particle systems and provide new dualities for several interacting systems in the continuum.
We also introduce a new process, the symmetric inclusion process in the continuum, for which our general method applies and yields generalized Meixner polynomials as orthogonal self-intertwiners.



\medskip\noindent
\emph{Acknowledgement:} S.F.\ acknowledges financial support from Netherlands Organisation for Scientific Research (NWO) through grant TOP1.17.019. S.J.\ and S.W.\ were supported under Germany's excellence strategy EXC-2111-390814868. S.J.\ and S.W.\ thank T.~Kuna and E.~Lytvynov for helpful discussions. F.R. thanks W.~Groenevelt for useful discussions and for pointing out the reference \cite{al1976convolutions}.
\end{abstract}

\tableofcontents
\bigskip

\footnoterule
\noindent
\hspace*{0.3cm} {\footnotesize $^{1)}$ 
Delft Institute of Applied Mathematics, TU Delft, Delft, The Netherlands,
{\sc s.floreani@tudelft.nl}}\\ 
\hspace*{0.3cm} {\footnotesize $^{2)}$ 
Mathematisches Institut, Ludwig-Maximilians-Universit\"{a}t, M\"{u}nchen, Germany,
{\sc jansen@math.lmu.de}}\\
\hspace*{0.3cm} {\footnotesize $^{3)}$ 
Delft Institute of Applied Mathematics, TU Delft, Delft, The Netherlands,
{\sc f.h.j.redig@tudelft.nl}\\
\hspace*{0.3cm} {\footnotesize $^{4)}$ 
	Mathematisches Institut, Ludwig-Maximilians-Universit\"{a}t, M\"{u}nchen, Germany,
	{\sc swagner@math.lmu.de}}}

\newpage

\section{Introduction}
\label{s.intro}

Duality and self-duality are important technical tools in the study of interacting particle systems and models of population dynamics.
E.g. in \cite{liggett_interacting_2005} self-duality is the key tool to analyze the ergodic properties of the symmetric exclusion process (SEP),
in \cite{kipnis1982heat} duality is the key tool to infer properties of a non-equilibrium steady state in the so-called KMP model of heat conduction and in \cite{DualitySpatiallyFlemingViot}, \cite{EthierKurtz} duality is used to study the long-time behaviour in stochastic models of population genetics.

Self-duality can be expressed as the property that the time evolution of well-chosen polynomials of degree $n$ boils down to the time evolution of $n$ (dual) particles. This is a significant simplification, because properties of a system of possibly infinitely many particles are reduced to properties of a finite number of particles. In its simplest setting (one particle duality in discrete systems), it is the property that the expected number of particles at a given location $x$ at time $t>0$ can be expressed in terms of the initial configuration and the location at time $t>0$ of {\em a single particle starting from $x$}.

In the setting of non-equilibrium systems, i.e., systems driven by boundary reservoirs where particles can enter and leave the system, duality with a system where the reservoirs are replaced by absorbing boundaries is a very useful and powerful tool
(see \cite{kipnis1982heat}, \cite{floreani2020orthogonal}, \cite{carinci2013duality}, \cite{derrida2001free}, \cite{frassek2021exact}, \cite{floreani2021switching}) because it expresses the $n$-point correlation functions in the non-equilibrium steady state (NESS) in terms of absorption probabilities of $n$ dual (i.e., absorbing) particles. By using one dual particle, one recovers easily the expected number of particles in the NESS, leading in the simplest setting of a linear chain to a linear density profile. For multiple dual particles, sometimes closed-form expressions can be inferred for their absorption probabilities, mostly when extra (integrability) properties are present. However, even if no closed-form expressions are available, many important properties such as local equilibrium and scaling properties of fluctuation fields  can still be inferred from  self-duality.

Duality and self-duality are very related properties. Indeed, for many systems such as independent random walkers (IRW) and the symmetric inclusion process (SIP), one can obtain duality by starting from self-duality and taking e.g.\ a many particle limit in the original particle number variables, while keeping the dual variables fixed.
This ``diffusion limit'' procedure yields e.g. dualities between independent random walkers and a deterministic system of coupled ordinary differential equations, as well as the duality between SIP and Brownian energy process (\cite{gkrv}).
Mostly dualities and self-dualities in the context of conservative interacting particle systems (such as the symmetric exclusion process) are studied on lattices, i.e., in a discrete setting where the particle configuration specifies at each site of the lattice the number of particles. The duality functions are then usually products over lattice sites of polynomials in the number of particles, depending on the number of dual particles (the number of dual particles corresponds to the degree of the polynomial).
The self-duality functions are usually categorized in ``classical'' duality functions, corresponding to (modified) factorial moments, and
``orthogonal'' self-duality functions which are products of orthogonal polynomials (where orthogonality is with respect to an underlying reversible product measure).
 
Self-duality with orthogonal polynomials is particularly useful in the study of fluctuation fields, the Boltzmann Gibbs principle, and cumulants in non-equilibrium steady states (\cite{DeMasiPresutti}, \cite{floreani2020orthogonal}, \cite{ayala2021higher}, \cite{ayala2018quantitative}). So far, for the classical discrete systems (SEP, SIP, IRW) orthogonal polynomial self-duality properties were obtained via various methods: the three term recurrence relations \cite{franceschini2019stochastic}, Lie algebra representation theory \cite{Groenevelt2019}, unitary symmetries \cite{carinci2019orthogonal}, and a direct relation between ``classical’’ (factorial moment) duality and orthogonal duality functions
\cite{floreani2020orthogonal} (i.e., an explicit form of Gram-Schmidt orthogonalization).

The language and formulation of duality in terms of occupation variables at discrete lattice sites clearly breaks down in many natural settings of e.g.\  particles moving in the continuum, such as interacting Brownian motions. 
Even in the simple context of independent Brownian motions, it is not immediately clear how to obtain self-duality. In particular, the naive approach of using the scaling limit of self-dualities of independent random walkers does not lead to useful results.
However, it is very natural to expect that all the classical discrete systems with self-duality properties have counterparts in the continuum. It is therefore important to develop a language in which the basic duality properties of discrete systems, including the orthogonal dualities, can be restated in such a way that they make sense in the continuum.
This is also important in view of studying scaling limits of systems which have self-duality properties such as the symmetric inclusion process in the condensation limit (where a system of sticky Brownian motions seems to appear), or the stochastic partial differential equations arising from the scaling limit of fluctuation fields. If these scaling limits share duality properties of their discrete approximants, this can be very useful in their analysis.

In \cite{carinci2021consistent}, the notion of consistency (see also \cite{kipnis1982heat}) was connected to self-duality, and in particular it was shown that for the three basic particle systems having self-duality (symmetric inclusion and exclusion process and independent random walkers), these dualities all can be derived from the same intertwining, which in turn is derived in a direct way from consistency. Consistency roughly means that the time evolution commutes with the operation of randomly selecting a given number of particles out of the system. Equivalently, up to permutations, it implies that in a system of $n$ particles, the $k$ particle evolution is coinciding with the evolution of $k$ particles out of these $n$ particles, i.e., the effect of the interactions with the other $n-k$ particles is ``wiped out’’. This is a remarkable property, trivially valid for independent particles, but also for interacting systems with special symmetries, such as the SEP and SIP. It is equivalent with a commutation property of the semigroup of the system with the ``lowering operator'' (an operator which describes essentially the removal of a random particle).

The consistency property clearly makes sense in a very general context, and appeared (under a slightly different form) in the literature on stochastic flows
\cite{LeJanRaimond}, \cite{schertzer_sun_swart_2016} including e.g.\ interacting Brownian motions, the Brownian web, and the Howitt-Warren flow. 
It also played a crucial role in the analysis of the KMP model (\cite{kipnis1982heat}).

Therefore, the consistency property seems the natural starting point for establishing dualities in a general context. 
The aim of our present paper is precisely this. We consider a very general setting of consistent Markovian evolutions on particle configurations and prove both ``classical'' (factorial moment) self-duality and ``orthogonal polynomial'' self-duality. More precisely we obtain a general ``factorial moment'' intertwining, and a general ``orthogonal polynomial'' intertwining. These general results first of all unify all the previous self-dualities of discrete conservative particle systems. Next, and more importantly,  they are valid in many new systems including independent and interacting diffusion processes, and a new process which we introduce here which is the natural continuum counterpart of the symmetric inclusion process.

Because we want to consider evolution of
configurations of particles, we are naturally led to the context of point processes (\cite{LastPenroseLecturesOnThePoissonProcess}). 
It turns out that in this language, the so-called classical self-dualities can be reformulated first as an intertwining, and they
relate very naturally to the so-called ``$n$-th factorial measures''. Therefore, the correct reformulation of ``classical self-duality'' in this general setting corresponds to intertwining with an intertwiner which is simply integration with respect to the ``$n$-th factorial measure''. It is then
a well-known fact (see e.g.\ Lemma 2.1 in Groenevelt \cite{Groenevelt2019}, in a discrete setup) that if an intertwiner can be written as a kernel operator in an $L^2(\mu)$ space where $\mu$ is a reversible measure, then
the corresponding kernel is a self-duality function. Therefore, the general ``$n$-th factorial measure'' intertwiner that we find indeed gives back all the known self-dualities in the discrete conservative particle systems.

Turning then to orthogonal polynomial self-duality, the situation becomes slightly more subtle in the general setting.
In \cite{franceschini2019stochastic} the authors remarked that the orthogonal self-duality polynomials can be obtained from Gram-Schmidt orthogonalization of the classical duality functions. In \cite{floreani2020orthogonal}, it is proved that this Gram-Schmidt orthogonalization in the context of discrete conservative particle systems,  is explicit and takes a rather simple form, from which it can be seen directly that it commutes with the semigroup of the particle evolution. This in turn implies that the Gram-Schmidt orthogonalization of the classical self-duality functions indeed produces a new self-duality function because a symmetry (i.e.\ an operator commuting with the semigroup) acting on a self-duality function produces a new self-duality function.
The simple form of the Gram-Schmidt orthogonalization in \cite{floreani2020orthogonal} does not generalize easily to the setting of point processes which we consider here. Therefore, how to obtain orthogonal intertwining from the factorial measure intertwiner remained an open problem. We solve this problem in a direct and simple way, by showing that from the factorial measure intertwiner combined with reversibility, one obtains that the Gram-Schmidt orthogonalization commutes with the semigroup. Therefore, one obtain in full generality a ``orthogonal factorial moment'' intertwiner. This intertwiner, when specified to discrete conservative particle systems, unites all orthogonal polynomial self-dualities.
Moreover, in the context of systems evolving in the continuum, it provides many new orthogonal self-dualities. In particular, in the continuum inclusion process introduced here, we find that our general orthogonal intertwiner when applied to specific functions
(tensor products of indicators of disjoint sets) becomes a product of Meixner polynomials in the number of particles.
In the context of independent particles, our general orthogonal intertwiner when applied to these specific functions gives the product of Charlier polynomials.

To summarize: the main two results of our paper (the $n$-th factorial measure intertwining and the corresponding orthogonal intertwiner)
firstly 
unify the previous `classical'' and ``orthogonal'' self-duality results in discrete setting and 
and secondly provide a much more general context in which self-duality results can be obtained an applied.

The rest of our paper is organized as follows.
In Section \ref{section: IRW} we revisit self-duality for independent
random walkers and link it to factorial moment measures of point processes.
This allows us to rewrite the self-duality relation in such a way that it makes sense for independent Markov processes on general state spaces, provided
a symmetry condition is fulfilled. In Section \ref{section: self-duality} we introduce the general setting and the class of Markov processes under consideration. We then state the two main theorems, the two self-intertwining results where the intertwiners are respectively, generalized falling factorial and orthogonal polynomials. In Section \ref{section: examples} we list some examples of known processes which satisfy the assumptions of our main theorems. In particular we show how the known self-duality relations for exclusion and inclusion process follow from our general results. In Section \ref{section: gSIP} we introduce and study a continuum version of the inclusion process. In particular we  identify  its reversible distribution, we show that it satisfies the assumptions of the two intertwiner results, and finally we exhibit the relation between the generalized orthogonal polynomials and the Meixner polynomials. Finally, in Appendix \ref{appendix: properties ort pol} we provide the proof of some properties of the generalized orthogonal polynomials introduced in Section \ref{section: Self-intertwining and orthogonality relations}

\section{Self-Duality for Independent Random Walkers on a Finite Set}
\label{section: IRW}
The section is mainly intended for the reader familiar with the language of interacting particle systems and ``interpolates'' between the usual notations of that context and the point process notations. The reader who is familiar with point process notation and wants to know the main results without the  interacting particle systems notational context can skip this section and pass to section \ref{section: self-duality}.

We start by considering a system of independent random walks on a finite set, for which
duality and self-duality properties  are well-known (see, e.g., \cite{gkrv, DeMasiPresutti}). The aim of the section is to revisit these duality results in the language of labelled particles. This will provide us with a notational framework in which these known duality relations are cast in a language that makes sense in a much more general setting, namely  independent Markov processes on a general state space. In this way, the reader is prepared in a convenient and easy case to the general framework that we build in Section \ref{section: self-duality}. 

Let $E$ be a finite set and $(\eta_t)_{t\ge 0}$, $\eta_t=(\eta_t(x))_{x\in E}$, be the Markov process on $\N_0^E$ generated by
\[
L f(\eta) = \sum_{x,y\in E} \eta(x) c(x,y) (f(\eta-\delta_x+\delta_y)-f(\eta)) 
\]
for $f:E\to \R$, $c:E\times E\to \R_+$ a symmetric function ($c(x,x)=0$ for any $x\in E$ without loss of generality) and where $\delta_x$ denotes the configuration with a single particle at $x$ and no particles at other locations. That is, $\delta_x$ is the configuration $\eta\in \N_0^E$ given by $\eta(y) = \delta_{x,y}$. Then $(\eta_t)_{t\ge 0}$ is called the configuration process with $\eta_t(x)$ denoting the numbers of particles at time $t\ge 0$ in $x\in E$. 
We denote by $p_t(x,y)$ the  transition probability of a single random walk, which is a symmetric function due to the symmetry of the rates $c:E\times E\to \R_+$.

Let $\xi\in \N_0^E$, we then define, the polynomials
\begin{equation}\label{dualitypol}
	D(\xi, \eta):= \prod_{x\in E} \frac{\eta(x)!}{(\eta(x)-\xi(x))!}\, \one_{\{\xi(x)\leq \eta(x)\}},\quad \eta\in \N_0^E.   
\end{equation}
We refer to \eqref{dualitypol} as the \textit{classical self-duality functions} and to $\xi$ as the dual configuration.  
The self-duality relation for the system of independent walkers then reads as follows
\begin{equation}
    \label{equation: selfdual}
    \E_\eta \left(D(\xi, \eta_t)\right)= \E_\xi \left(D(\xi_t, \eta)\right)
\end{equation}
for all $\eta,\xi\in \N_0^E$ and $t\ge 0$, where $\E_\eta$ denotes the expectation in the configuration process started from $\eta$. Notice here that we have restricted to the case of finite $E$ for convenience but \eqref{equation: selfdual} can be extended to countable $E$, a suitable set of allowed starting configurations $\eta\in \N_0^E$, and finite dual configurations $\xi$. In Section \ref{section: self dual 1 dual part IRW}, by a change of notation, we reformulate the relation \eqref{equation: selfdual} with one dual particle (i.e., $\xi=\delta_x$) in such a way that it is meaningful in contexts more general than random walks on a finite set, namely also in the continuum. Thus we get rid of the configuration process notation. In Section \ref{section: n dual part IRW} we proceed by reformulating \eqref{equation: selfdual} in the general case with $n$ dual particles and finally, in Section \ref{section: orthogonal self-duality RW}, we recall and reformulate the so-called orthogonal self-dualities for independent random walks.

\subsection{The Labeled Configuration Notation and the Associated Point Configuration: Self-Duality with One Dual Particle}
\label{section: self dual 1 dual part IRW}

Let $\mathcal X:= (\mathcal X_0(1),\ldots, \mathcal X_0(N))$  be an arbitrary labelling of the initial positions of the particles which are in total $N<\infty$.  We then denote $\mathcal X_t$ the positions of these particles at time $t\geq 0$, with $\mathcal X_t(i)$ the position of the $i$-th particle at time $t\geq 0$. 
The correspondence between the labeled system $(\mathcal X_t)_{t\ge 0}$ and the previously introduced configuration process $(\eta_t)_{t\ge 0}$ is given by $\eta_t(x)= \sum_{i=1}^N \one_{\{\mathcal X_t(i)=x\}}$.

We describe the system also via the point configuration $\sum_{i=1}^N \delta_{\mathcal X_t(i)}$. Notice that in this discrete setting, this is simply a change of notation for the configuration: indeed, for $x\in E$, we have $\left(  \sum_{i=1}^N \delta_{\mathcal X_t(i)}\right)(\{x\})= \eta_t(x)$. In view of the generalization of self-duality in the next sections, from now on, we identify $\eta_t$ with the point configuration
$$\eta_t
= \sum_{i=1}^N \delta_{\mathcal X_t(i)}.$$
This is the same as identifying a measure $\eta$ on the finite $E$ with the vector 
$\eta(\{x\}), x\in E$. The advantage of this change of notation is that it generalizes to arbitrary measurable state spaces $E$, and
it also allows to produce a simple but insightful proof of the self-duality \eqref{equation: selfdual}.

Let us start with self-duality with a single dual particle, i.e.~\eqref{dualitypol} with dual configuration $\xi=\delta_x$, which reads as

\begin{equation*}
    \E_\eta (\eta_t(\{x\}))=\ERW_x\left( \eta_0({\{Y_t\}})\right)=\sum_{y\in E} p_t(x,y)\eta(\{y\}),
\end{equation*}
where $\ERW_x$ denotes the expectation with respect to random walk with transition rates $c(x,y)$ starting at $x\in E$. 

Let us denote by $\E_{\mathcal{X}} (\eta_t)$ the measure defined as $\E_{\mathcal X}\bra{\eta_t}(A) := \E_{\mathcal X} \bra{\eta_t(A)}$ for $A \subset E$, where $\E_{\mathcal X}$ denotes the expectation when starting $\mathcal X_t(1), \ldots, \mathcal X_t (N)$ at $\mathcal X$.  We then have 
\begin{equation}\label{bokno}
    \E_{\mathcal X}\left( \eta_t\right)= \E_{\mathcal X}\left(\sum_{i=1}^N \delta_{\caX_t(i)}\right)= \sum_{i=1}^N\E_{\caX}( \delta_{\caX_t(i)})
    = \sum_{i=1}^N\ERW_{\caX_0(i)} (\delta_{\caX_t(i)})
\end{equation}
where in the third equality in \eqref{bokno} we used that the particles are independent, i.e., the distribution of the position of the $i$-th particle is only depending on its initial position $\mathcal X_0(i)$ and not on the other particles. Using that $\ERW_{\caX_0(i)} (\delta_{\caX_t(i)})=\sum_{y\in E}p_t(\caX_0(i),y) \delta_y$, we can rewrite \eqref{bokno} as
\begin{align*}
    \E_{\mathcal X}\left( \eta_t\right)
    &= \sum_{i=1}^N\sum_{y\in E}p_t(\caX_0(i),y) \delta_y=\sum_{y\in E} \delta_y  \sum_{i} p_t(y,\caX_0(i)) 
    =\sum_{y\in E}  \left(\int p_t(y,z)\eta_0(\dd z) \right)  \delta_y
\end{align*}
where in the fourth equality we used the symmetry of the transition probabilities $p_t(x,y)$. If we denote by $\lambda(\dd y)$ the counting measure on $E$ we obtain
\begin{equation}
    \label{eq: self-duality one dual particle IRW}
    \left(\E_{\caX} \left(\eta_t\right)\right)(\dd y)=  \left(\int p_t(y,z)\eta_0(\dd z) \right) \lambda(\dd y).
\end{equation}
The above reformulation of the self-duality relation with one dual particle now makes sense on general measurable state spaces $E$.

\subsection{Reformulation of Self-Duality with \texorpdfstring{$n$}{n} Dual Particles}
\label{section: n dual part IRW}
As a next step we want to generalize \eqref{eq: self-duality one dual particle IRW}
to the case of  $n$ dual particles. In order to do so, given a point configuration $\eta=\sum_{i=1}^N\delta_{x_i}$, we introduce the $n$-th factorial measure of $\eta$ (see, e.g., \cite[Eq.~(4.5)]{LastPenroseLecturesOnThePoissonProcess}), which is given by
\begin{equation}
    \label{eq: falling fact meas}
    \eta^{(n)} = \sideset{}{^{\neq}}\sum_{1 \leq i_1, \ldots, i_n \leq N} \delta_{(x_{i_1}, \ldots, x_{i_n})}.
\end{equation}
Using the  notation adopted in \cite{LastPenroseLecturesOnThePoissonProcess}, the superscript $\neq$ indicates summation over $n$-tuples with pairwise different entries and where an empty sum is defined as zero. The reason why the measure in \eqref{eq: falling fact meas} is called falling factorial is clearly explained by the elementary combinatorial lemma below, where the relation with the classical dualities defined in \eqref{dualitypol} (consisting of products of falling factorial polynomials) is  given. We leave the simple proof to the reader.

\begin{lemma}
    Let $\eta=\sum_{i=1}^N\delta_{x_i}$. Then, for all $(y_1,\ldots,y_n)\in E^n$,  we have
    \begin{equation}\label{simpi}
        \eta^{(n)}(\{(y_1,\ldots,y_n)\})=
        D\left(\sum_{k=1}^n\delta_{y_k}, \eta\right),
    \end{equation}
    where $D(\sd,\sd)$ is the self-duality polynomial function given in \eqref{dualitypol}. As a consequence, the $n$-th factorial measure  can be rewritten as follows
    \begin{equation}
        \label{combid}
        \eta^{(n)} = \sum_{y_1, \ldots, y_n\in E} \delta_{(y_1,\ldots, y_n)} D\left(\sum_{k=1}^n\delta_{y_k}, \eta\right).
    \end{equation}
\end{lemma}

We can then generalize \eqref{eq: self-duality one dual particle IRW} to the expectation of the $n$-th factorial measure measure $\eta^{(n)}_t$ of  the point configuration valued process $\eta_t=\sum_{i} \delta_{\mathcal X_t(i)}$ introduced above. 

\begin{proposition}
    \label{gensdt}
	Let $\lambda$ be the counting measure on $E$. Then, for all $t>0$ and $n\in \N$,
	\begin{equation}
	    \label{vivi}
		\E_{\caX} (\eta^{(n)}_t)(\dd(y_1 \ldots y_n))= \left(\int_{E^n} \prod_{i=1}^n p_t(y_i, z_i) \eta_0^{(n)}( \dd(z_1, \ldots z_n))\right) \lambda^{\otimes n}(\dd(y_1, \ldots, y_n))\\
	\end{equation}
\end{proposition}

\begin{proof}
    Let $f:E^n\to\R$. We then have 
    \begin{align}
        &\label{eq: rewriting fall fact}\E_{\caX}\left(\int f(y_1, \ldots, y_n) \eta^{(n)}_t(\dd(y_1 \ldots y_n))\right)
        =
        \sideset{}{^{\neq}}\sum_{1 \leq i_1, \ldots, i_n \leq N}\E_{\caX} f(\caX_{t}(i_1), \ldots, \caX_{t}(i_n))
        \\
        &=\nonumber
        \sideset{}{^{\neq}}\sum_{1 \leq i_1, \ldots, i_n \leq N}
        \int f(y_1, \ldots, y_n) \prod_{k=1}^n p_t (\caX_{0}(i_k), y_k) \prod_{k=1}^n \lambda(\dd y_k)
        \nonumber\\
        &=\nonumber
        \sideset{}{^{\neq}}\sum_{1 \leq i_1, \ldots, i_n \leq N}
        \int f(y_1, \ldots, y_n) \prod_{k=1}^n p_t (y_k, \caX_{0}(i_k)) \prod_{k=1}^n \lambda(\dd y_k)
        \nonumber\\
        &=\nonumber
        \int f(y_1, \ldots, y_n) \left(\int\prod_{k=1}^n p_t (y_k, z_k)\eta_0^{(n)}(\dd(z_1\ldots z_n))\right) \prod_{k=1}^n \lambda(\dd y_k).
    \end{align}
    where we used \eqref{eq: falling fact meas} in the first and the last equality, the independence of the particles in the second equality and the of the transition probabilities in the third equality.
    Because $f$ is arbitrary, this proves \eqref{vivi}
\end{proof}

\begin{remark}
    \begin{enumerate}[\normalfont(i)]
        \item Equation~\eqref{vivi} holds for each system of independent reversible random walks where the reversible measure  $\lambda_{\rm{rev}}$ for the single random walk  is used in place of the counting measure $\lambda$. 
        \item Without assuming the symmetry of the rates $c:E\times E\to \R$, from \eqref{eq: rewriting fall fact} and the independence of the particles, we still have the relation 
        \begin{equation}
            \label{eq: intert sec 2}
            \E_{\caX}\left(\int_{E^n} f \dd \eta^{(n)}_t   \right)=\int_{E^n} \E_{y_1,\ldots,y_n}^{\rm{IRW}}\left(f(Y_t(1),\ldots,Y_t(n))  \right)  \eta_0^{(n)}(\dd(y_1\ldots,y_n)),
        \end{equation} 
        where $f:E^n\to\R$ is a permutation invariant function and $\ERW_{y_1, \ldots,y_n}$ denotes expectation with respect to $n$ independent random walkers initially starting from
        $(y_1, \ldots, y_n)$. Equation~\eqref{eq: intert sec 2} has to be read as a self-intertwining relation and it will be generalized in Section \ref{section: generalized falling factorial polynomials}.
        \item[iii)] For any $(y_1,\ldots,y_n)\in E^n$, \eqref{vivi} implies
        $$ \E_{\mathcal X}(\eta_t^{(n)}(\{(y_1,\ldots,y_n)\} )=\ERW_{y_1, \ldots,y_n}\left(\eta^{(n)}_0(\{(Y_t(1),\ldots,Y_t(n))\})\right)$$
        which, in view of \eqref{simpi}, reads as 
        $$\E_\caX \left(D\left(\sum_{k=1}^n\delta_{y_k}, \eta_t\right)\right)=\ERW_{y_1, \ldots,y_n}\left(D\left(\sum_{k=1}^n\delta_{Y_t(i)}, \eta\right)\right) ,$$
        which is precisely the \textit{classical self-duality} relation given in  \eqref{equation: selfdual}.
    \end{enumerate}
\end{remark}

\subsection{Orthogonal Self-Duality}

\label{section: orthogonal self-duality RW}

In this section we turn to orthogonal self-dualities for random walks in a finite set. In \cite{redig_factorized_2018}, \cite{franceschini2019stochastic} and \cite{floreani2020orthogonal} it has been shown (using, respectively, generating functions method, three term recurrence relations and algebraic methods) that, for all $\theta> 0$, the following self-duality relation holds
\begin{equation}
    \label{oselfdual}
    \E_\eta(D_\theta(\xi,\eta_t))=\E_\xi(D_\theta(\xi_t,\eta))
\end{equation}
with respect to the self-duality functions
\begin{align}
    \label{eq: ortduality}
    D_\theta(\xi, \eta)= \prod_{x\in S}d^{\rm{or}}_{\xi(\{x\})}(\eta(\{x\});\theta).
\end{align}
$\{ d^{\rm{or}}_n(\sd; \theta)\}_{n\in\N}$ are the Charlier polynomials, i.e.\ the polynomials satisfying  the following orthogonality relation
$$\int  d^{\rm{or}}_n(\eta(\{x\}); \theta)  d^{\rm{or}}_m(\eta(\{x\}); \theta)  \rho_{\theta}(\dd \eta)= \one_{\{n=m\}} \frac{n!}{\theta^n}$$
with $\rho_\theta=\otimes_{x\in E} \rho_{x, \theta}$ and $\rho_{x, \theta}=\Poi(\theta)$ for each $x\in S$. We refer to the functions in \eqref{eq: ortduality} as \textit{orthogonal self-dualities}. Let $[n]:=\{1,\ldots,n\}$ In this setting, the relation between orthogonal and classical dualities is simple and  given by (see \cite[Remark 4.2]{floreani2020orthogonal})

\begin{align}
    \label{eq: ort dual I}
    D_\theta(\xi, \eta)=\sum_{\xi'\le\xi}(-\theta)^{|\xi|-|\xi'|}\binom{\xi}{\xi'} D(\xi',\eta) =\sum_{I\subset [n]}(-\theta)^{n-|I|}D\bra{\sum_{i\in I}\delta_{y_i},\eta}
\end{align}
from which it follows that \eqref{oselfdual} is a direct consequence of  \eqref{equation: selfdual} and the independence of the particles. 
We can now reformulate the self-duality relation \eqref{oselfdual} in terms of a point configuration notation. First we introduce the orthogonalized version of the falling factorial measure associated to a point configuration $\eta=\sum_{i=1}^N\delta_{x_i}$, namely
\begin{equation}
    \eta^{(n),\theta}(\dd (x_1, \ldots, x_n)):=\sum_{r=0}^n (-\theta)^{n-r} \sum_{I\subset [n]: |I|=r}\eta^{(r)}(\dd (x_1, \ldots, x_n)_I)\otimes\lambda ^{\otimes (n-r)}(\dd (x_1, \ldots, x_n)_{[n]\setminus I}),
\end{equation}
where $\lambda$ denotes the counting measure,
$\int f_0 \dx \eta^{(0)} := f_0$ for all $f_0 \in \R$ and  $(x_1, \ldots, x_n)_I$ denotes the subvector of $(x_1, \ldots, x_n)$ with components in $I\subset [n]$. The relation between $\eta^{(n),\theta}$ and the orthogonal self-dualities is expressed in the following result.

\begin{lemma}
    Let $\eta=\sum_{i=1}^N\delta_{x_i}$. Then, for all $(y_1,\ldots,y_n)\in E^n$,  we have
    \begin{equation}\label{simpi2}
        \eta^{(n),\theta}(\{ (y_1,\ldots,y_n) \})=D_\theta\bra{\sum_{i=1}^n \delta_{y_i},\eta}
    \end{equation}
    where $D_\theta(\sd,\sd)$ is the orthogonal self-duality given in \eqref{eq: ort dual I}. As a consequence
    \begin{equation*}
        \eta_t^{(n),\theta}=\sum_{y_1,\ldots,y_n\in E}D_\theta\bra{\sum_{i=1}^n \delta_{y_i},\eta}\delta_{(y_1,\ldots,y_n)}.
    \end{equation*}
\end{lemma}

\begin{proof}
    For $I\subset [n]$ with $|I|=r$, we have, using \eqref{simpi},
    \begin{align*}
        & D\bra{\sum_{i\in I} y_i, \eta} = \eta^{(r)}((y_1,\ldots,y_n)_I)=\int \one_{(y_1,\ldots,y_n)_I}(x_1,\ldots,x_r) \eta^{(r)}(\dd (x_1,\ldots,x_r))\\
        &=  \int \one_{(y_1,\ldots,y_n)}(x_1,\ldots,x_n) \eta^{(r)}(\dd(x_1,\ldots,x_n)_I) \otimes\lambda ^{ \otimes (n-r)}(\dd (x_1,\ldots,x_n)_{[n]\setminus I}).
    \end{align*}
    Therefore, \eqref{simpi2} follows from \eqref{eq: ort dual I}.
\end{proof}

We then state the analogue of Proposition~\ref{gensdt} for $\eta^{(n), \theta}$ in a notation which makes sense in the context of general measurable state space $E$. The result follows from \eqref{vivi} combined with the definition of $\eta^{(n), \theta}$ and the reversibility of $\lambda$ for the single random walk: we omit here the simple proof  and we refer to Section \ref{section: Self-intertwining and orthogonality relations} for the proof of the self-intertwining formulation of this result in a much more general setting.

\begin{proposition}
    For all $t>0$ and $n\in\N$
    \begin{equation}
        \label{eq: duality ort pol general}
        \E_{\mathcal X}(\eta^{(n),\theta}_t) (\dd(y_1, \ldots, y_n)) = \left(\int_{E^n} \prod_{i=1}^np_t(y_i,x_i)\eta^{(n),\theta}_0(\dd (x_1,\ldots, x_n)) \right) \lambda^{\otimes n}(\dd(y_1, \ldots, y_n)).
    \end{equation}
\end{proposition}

It was observed in \cite{franceschini2019stochastic} (just above equation (8) in \cite{franceschini2019stochastic}), that the orthogonal self-dualities given in \eqref{eq: ortduality} coincide with the polynomials  obtained by the Gram-Schmidt orthogonalization procedure initialized with the classical duality functions given in \eqref{dualitypol}. In the present context, the Gram-Schmidt orthogonalization applied to \eqref{dualitypol} is \eqref{eq: ort dual I}. However, so far, no proof was provided of the fact that the orthogonalization procedure applied to classical self-duality functions leads again  to self-duality functions. In Section \ref{section: Self-intertwining and orthogonality relations} below, we prove, in a much more general context, that if we properly orthogonalize 
a self-intertwiner which is a generalized falling factorial polynomial, we get a generalized orthogonal polynomial which is again a self-intertwiner. The proof  boils down  to show the commutation of the semigroup of the point configuration process with the linear map of the orthogonalization procedure, i.e. that the orthogonalization procedure is a symmetry. From the self-intertwining relations just mentioned follows both classical and orthogonal self-duality relations. 
The self-intertwiner related to the generalized falling factorial polynomials is introduced in Section \ref{section: generalized falling factorial polynomials} below and  the connection between self-intertwining and classical self-dualities is explained in Section \ref{section: example IPS on finite set}.

\section{Self-Intertwining Relations}
\label{section: self-duality} 
In this section, we start by introducing the setting and the class of processes that we consider, namely the consistent and conservative Markov processes. Then, in Section \ref{section: generalized falling factorial polynomials}, we introduce the generalized falling factorial polynomials and we state and prove our first main result, a self-intertwining relation. In Section \ref{section: Self-intertwining and orthogonality relations}, after providing the construction of generalized orthogonal polynomials, we state and prove a second self-intertwining relation.

\subsection{Setting and Consistent Markov Processes}
\label{section:setting}

Throughout this article we investigate Markov processes whose state space consists of configurations of  non-labelled particles in some general measurable space $(E,\mathcal E)$. To avoid the technical difficulties associated with infinitely many particles (for example, a rigorous construction of interacting dynamics), we consider configurations of finitely many particles only.

We follow modern point process notation in modelling such configurations as finite counting measures on $(E,\mathcal E)$. Thus, let $\mathbf N_{<\infty}$ be the space of finite counting measures, i.e., the space of finite measures that assign values in $\N_0$ to every set $B\in \mathcal E$. The space is equipped with the $\sigma$-algebra $\mathcal N_{<\infty}$ generated by the counting variables $\mathbf N_{<\infty}\ni \eta\mapsto \eta(B)$, $B\in \mathcal E$. Assumptions on $(E,\mathcal E)$ are needed to ensure that every counting measure is a sum of Dirac measures, therefore we assume throughout the article that $(E,\mathcal E)$ is a Borel space (see \cite[Definition 6.1]{LastPenroseLecturesOnThePoissonProcess}). The reader  may think of a Polish space or $\R^d$. It is well-known (see, e.g., \cite[Chapter 6]{LastPenroseLecturesOnThePoissonProcess} or \cite[Section 1.1]{Kallenberg2017}) that for a Borel space, every finite counting measure $\eta\in \mathbf N_{<\infty}$ is either zero or of the form $\eta = \delta_{x_1}+\cdots +\delta_{x_n}$ for some $n\in \N$ and $x_1,\ldots, x_n\in E$. 
In particular, the total mass $\eta(E)$ corresponds to the total number of particles. 

For our purpose, a Markov process with state space $\mathbf{N}_{<\infty}$ is a collection
$(\Omega, \mathfrak F,(\eta_t)_{t\ge 0}, (\P_\eta)_{\eta \in \mathbf{N}_{<\infty}} )$, where $(\Omega, \mathfrak F)$ is a measurable space, $\eta_t:(\Omega, \mathfrak F)\to (\mathbf{N}_{<\infty}, \mathcal N_{<\infty})$ is a measurable map for all $t\geq 0$ and for $\eta \in \mathbf{N}_{<\infty}$, $\P_\eta$ are probability measures on $(\Omega, \mathfrak F)$ such that $\P_\eta(\eta_0=\eta) = 1$. The Markov property is implicitly assumed to be satisfied with respect to the natural filtration $\mathfrak F_t:=\sigma(\eta_s, \ 0\le s\le t)$.

We focus on a special class of Markov processes, which has been  considered in \cite{carinci2021consistent,LeJanRaimond, kipnis1982heat,schertzer_sun_swart_2016}, namely consistent Markov processes. Intuitively speaking, consistency refers to the fact that the removal of a particle uniformly at random commutes with the time-evolution of the process.
In order to precisely define the concept of consistent Markov process we introduce the lowering operator 
\begin{equation*}
    \mathcal{A}f(\eta) := \int f(\eta - \delta_x) \eta(\dd x), \quad \eta \in \mathbf{N}_{<\infty}
\end{equation*} 
acting on functions $f \in \mathcal{G}$, where $\mathcal{G}$ denotes the set of measurable functions $f :\mathbf{N}_{<\infty} \to \R$ such that the restriction of $f$ to every $n$-particle sector $\mathbf N_n:=\{\eta \in \mathbf N_{<\infty}:\ \eta(E) = n\}$ is bounded. Note  $\mathcal{A}$ is well-defined and that $\mathcal{A} f \in \mathcal{G}$ for $f \in \mathcal{G}$. 

\begin{definition}[Consistent Markov process]
    \label{def: consistent particle system}
    Let $(\eta_t)_{t\ge 0}$ be a Markov process  on $\mathbf{N}_{<\infty}$ with Markov semigroup $(P_t)_{t\geq 0}$. The process $(\eta_t)_{t\geq 0}$ said to be consistent if 
    for all $t\ge 0$ and  bounded and measurable function $f : \mathbf{N}_{<\infty} \to \R$ 
    \begin{align}\label{eq: consistency}
    	P_t\mathcal Af(\eta) = \mathcal A P_t f(\eta), \quad \eta \in \mathbf{N}_{<\infty}.
    \end{align}
\end{definition}

Notice that \eqref{eq: consistency} can be written as
\begin{align*}
    \E_{\eta}\bra{ \int f(\eta_t - \delta_x) \eta_t(\dd x)} = \int \E_{\eta - \delta_x} (f(\eta_t)) \eta(\dd x),
\end{align*}
where on the left hand-side we first evolve the system and after we remove uniformly at random a particle, while on the right hand-side we first remove uniformly at random a particle from the initial configuration and then we let evolve the process from the new initial state. We refer to 
\cite[Theorem~2.7 and Theorem~3.2]{carinci2021consistent} for further characterizations of consistency in terms of the infinitesimal generator $L$, namely $L\mathcal A= \mathcal AL$, and to Section~\ref{section: examples} and \ref{section: gSIP} for some examples of consistent Markov processes.

For our results we need the following set of assumptions.

\begin{assumption}
    \label{assumption: consistency and conserve number particles}
    We assume that $(\eta_t)_{t\geq 0}$ is a Markov process on $\mathbf{N}_{<\infty}$ with Markov semigroup $(P_t)_{t\geq 0}$, such that
    \begin{enumerate}[\normalfont(i)]
    	\item \label{item: consistency}it is consistent;
    	\item \label{item: conserving number particles} it is conservative, i.e. if $\eta_0\in \mathbf{N}_{<\infty}$ then $\eta_t( E )=\eta_0(E)$ for all $t\ge 0$. 
    \end{enumerate}
\end{assumption}

Notice that Assumption~\ref{assumption: consistency and conserve number particles}~(\ref{item: conserving number particles}) yields $P_t f \in \mathcal{G}$ for all $f \in \mathcal{G}$ and thus, by Assumption~\ref{assumption: consistency and conserve number particles}~(\ref{item: consistency}), we obtain $P_t \mathcal{A} f(\eta) =  \mathcal{A} P_t f(\eta)$ for $f \in \mathcal{G}$ and $\eta \in \mathbf{N}_{<\infty}$. 

Let us briefly explain how consistency as defined in Definition~\ref{def: consistent particle system} relates to a stronger form of consistency reminiscent of Kolmogorov's consistency theorem. Often the process $(\eta_t)_{t\geq 0}$ comes from a process for \emph{labelled} particles, as is the case for the independent random walkers in Section~\ref{section: IRW}. \emph{Strong consistency}, called \emph{compatibility} by Le Jan and Raimond 
\cite[Definition~1.1]{LeJanRaimond},  roughly means that time evolution and removal of any \emph{deterministic} particle commute---there is no need to choose the particle to be removed uniformly at random.

Precisely, suppose that for each $n\in \N$, we are given a transition function 
$(p_t^{[n]})_{t\geq 0}$ on $(E^n,\mathcal E^n)$ that preserves permutation invariance. Then one can define a transition function $(P_t)_{t\geq0}$ on $(\mathbf N_{<\infty}, \mathcal N_{<\infty})$ by $P_t(0,B)= \one_B(0)$ and 
\begin{equation} \label{eq:forgetting-labels}
	P_t(\delta_{x_1}+\cdots + \delta_{x_n},B)= p_t^{[n]}\bigl(x_1,\ldots,x_n; \iota_n^{-1}(B)\bigr),\quad (x_1,x_2,\ldots,x_n)\in E^n, B\in \mathcal N_{<\infty}
\end{equation} 
where  $\iota_n:E^n\to \mathbf N_{<\infty}$  is the map given by $\iota_n(x_1,\ldots,x_n) = \delta_{x_1}+\cdots + \delta_{x_n}$. 

\begin{definition}
    \label{def:strong-consistency}
	The family $(p_t^{[n]})_{t\geq 0}$ is \emph{strongly consistent} if for all $n\in \N$, $i\in \{1,\ldots,n\}$, and $(x_1,\ldots,x_n)\in E^n$, the image of the measure $\mathcal E^n\ni B\mapsto p_t^{[n]}(x_1,\ldots,x_n; B)$ under the map from $E^n$ to $E^{n-1}$ that consists of omission of $x_i$ is equal to the measure $\mathcal E^{n-1}\ni B\mapsto p_t^{[n-1]}(x_1,\ldots,\widehat x_i,\ldots, x_n;B)$, where $\widehat x_i$ means omission of the variable $x_i$.
\end{definition} 

An elementary but important observation is that strong consistency of the family $(p_t^{[n]})_{t\geq 0}$ implies consistency of $(P_t)_{t\geq0}$ in the sense of Definition~\ref{def: consistent particle system}. 
The observation yields a whole class of consistent processes, see Section~\ref{section:ex-strongly-consistent}.

Theorem~\ref{theorem: classical self-intertwining} uses both $(P_t)_{t\geq 0}$ and a semigroup $(p_t^{[n]})_{t\geq 0}$ for labelled particles. As we wish to use the semigroup $(P_t)_{t\geq 0}$ as our starting point, let us mention that \eqref{eq:forgetting-labels} implies 
\begin{equation}
    \label{eq:pt-to-ptn}
	(P_t f)(\delta_{x_1}+\cdots+ \delta_{x_n}) = \bigl( p_t^{[n]} f_n\bigr)(x_1,\ldots, x_n)
\end{equation}
whenever $f_n = f\circ \iota_n$ and $f:\mathbf N_{<\infty}\to \R$ is measurable and non-negative or bounded. This determines the action of $(p_t^{[n]})_{t\geq 0}$ on the space $\mathcal F_n$ of bounded, measurable, permutation invariant functions $f_n$ uniquely. Therefore, given a conservative semigroup $(P_t)_{t\geq 0}$ on $\mathbf N_{<\infty}$ we may take \eqref{eq:pt-to-ptn} as the \emph{definition} of an associated semigroup on the space of bounded permutation invariant functions $\mathcal F_n$. For $n = 0$, we set $\mathcal{F}_0 := \R$ and let $p_t^{[0]}$ be the identity operator on $\R$, for all $t \geq 0$.

\subsection{Generalized Falling Factorial Polynomials}
\label{section: generalized falling factorial polynomials}

Let $\eta=\sum_{i=1}^m \delta_{x_i}\in\mathbf{N}$, $n \in \N$, and recall (see \eqref{eq: falling fact meas} above) that $\eta^{(n)}$ denotes the $n$-th factorial measure  of $\eta$, i.e.
\begin{align*}
    \eta^{(n)}:= \sideset{}{^{\neq}}\sum_{1 \leq i_1, \ldots, i_n \leq m} \delta_{(x_{i_1}, \ldots, x_{i_n})},
\end{align*}
where $\eta = 0$ yields $\eta^{(r)} = 0$. 

\begin{definition}
    For $n\in \N$ and measurable $f_n:E^n \to \R$ we define the associated generalized falling factorial polynomial as follows
    \begin{align*}
        J_n(  f_n, \eta):= \int  f_n(x_1, \ldots, x_n) \eta^{(n)}(\dd (x_1, \ldots, x_n)), \quad \eta \in \mathbf{N}_{<\infty}.
    \end{align*}
    For $n = 0$ and $f_0 \in \R$ we set $J_0(f_0, \eta) := \int f_0 \dx \eta^{(0)} := f_0$. 
\end{definition}

In particular, we have $J_n(f_n, \sd) \in \mathcal{G}$ for $f_n \in \mathcal{F}_n$. 

\begin{remark}
    The fact that $J_n$ generalizes falling factorial polynomials becomes evident when considering $f_n=\one_{B_1}^{\otimes d_1} \otimes \cdots \otimes \one_{B_N}^{\otimes d_N}$ for pairwise disjoint sets $B_1, \ldots, B_N \in \mathcal{E}$, $N \in \N$ and $d_1 \ldots, d_N \in \N_0$, $d_1 + \ldots + d_N =: n$. Indeed, it follows from the definition of the factorial measure  that
    \begin{align}\label{eq: relation factorial from tensor products}
        J_n (\one_{B_1}^{\otimes d_1} \otimes \cdots \otimes \one_{B_N}^{\otimes d_N}, \eta) =  (\eta(B_1))_{d_1} \cdots (\eta(B_N))_{d_N}, \quad \eta \in \mathbf{N}_{<\infty}
    \end{align}
    where $(a)_k := a (a-1) \cdots (a-k+1)$, $a \in \R$, $k \in \N$, $(a)_0 := 1$,  denotes the falling factorial. Equation~\eqref{eq: relation factorial from tensor products} will be used in Section \ref{section: examples} below to recover known self-duality functions for particle systems on finite set from the abstract Theorem \ref{theorem: classical self-intertwining}. We refer to \cite{finkelshtein2021stirling} for further properties of the generalized falling factorial polynomials.
\end{remark}

Our first main result is an intertwining relation between the Markov semigroup $(P_t)_{t\geq 0}$ and $(p_t^{[n]})_{t\geq 0}$, with  the generalized falling factorial polynomials $J_n$ defined above as intertwiner. Thus, we view the result as a generalization of the self-duality relations for interacting particle systems on a finite set where the self-duality functions consist in weighted falling factorial moments of the occupation variables (see, e.g., \cite[Theorem 1.1, p.363]{liggett_interacting_2005}, \eqref{dualitypol} above and Section \ref{section: example IPS on finite set} below) .

\begin{theorem}[Self-intertwining relation]\label{theorem: classical self-intertwining} 
Let $(\eta_t)_{t\geq 0}$ be a Markov process satisfying Assumption~\ref{assumption: consistency and conserve number particles}. We then have
	\begin{align}
	    \label{eq: self intertwining Jr}
	    P_t J_n( f_n,\sd)(\eta) = J_n ( p_t^{[n]} f_n ,\eta), \quad \eta \in \mathbf{N}_{<\infty}
	\end{align}
	for each $f_n\in\mathcal F_{n}$, $n \in \N_0$ and $t \geq 0$.
\end{theorem}

\begin{proof}
    Let us define the lowering operator $\mathcal{A}_{r-1, r}$ acting on functions $f_{r-1}\in \mathcal F_{r-1}$ as
    \begin{align*}
        \mathcal{A}_{r-1, r} f_{r-1}(x_1, \ldots, x_r) := \sum_{k=1}^r f_{r-1}(x_1, \ldots, x_{k-1}, x_{k+1}, \ldots, x_r)
    \end{align*}
    for $x_1,\ldots,x_r\in E$ and $r \geq 2$ and $\mathcal{A}_{0,1} f_0 :=  f_0\one$, $f_0 \in \R$ for $r = 1$. We then have, as a direct consequence of consistency of $(\eta_t)_{t\geq 0}$, that $p_t^{[r]}  \mathcal{A}_{r-1, r} f_{r-1} =  \mathcal{A}_{r-1, r} p_t^{[r-1]} f_{r-1}$, $r \in \N$. Denoting for all $r \geq n \geq 0$,
    \begin{align*}
        \mathcal{A}_{n, r} f_n := \begin{cases} \mathcal{A}_{n,n+1} \cdots \mathcal{A}_{r-1, r} f_n & r >n \\
        f_n & n = r
        \end{cases},
    \end{align*}
    for all $f_n \in \mathcal{F}_n$, one obtains, by induction, that
    \begin{align*}
        p_t^{[r]}  \mathcal{A}_{n, r} f_{n} =  \mathcal{A}_{n, r} p_t^{[n]}f_n.
    \end{align*}
    The proof is concluded by noticing that for all $n\le r$, $x_1, \ldots, x_r \in E$, 
    \begin{equation*}
        J_n(  f_n, \delta_{x_1} + \ldots + \delta_{x_r})=\frac{n!}{(r-n)!}  \mathcal{A}_{n, r} f_n(x_1, \ldots, x_r). \qedhere
    \end{equation*} 
\end{proof}

\begin{remark}
    \label{rem:consistency-equivalence}
	A close look at the proof reveals that the relation in Theorem~\ref{theorem: classical self-intertwining} is in fact an equivalence: A conservative process is consistent if and only if the self-intertwining relation~\eqref{eq: self intertwining Jr} holds true for all $n,f_n,t$. The equivalence is closely related to Theorem~4.3
	in \cite{carinci2021consistent} in the discrete setting.
\end{remark}

Theorem~\ref{theorem: classical self-intertwining} can be rephrased in a number of ways. The first rephrasing is in terms of kernels and justifies the denomination \emph{intertwining}. Let $\Lambda_n: \mathbf N_{<\infty}\times \mathcal E^n\to \R_+$ be the kernel given by $\Lambda_n(\eta,B):= \eta^{(n)}(B) = J_n(\one_B,\eta)$. Then, $P_t \Lambda_n = \Lambda_n p_t^{[n]}$ meaning that 
\[
	\int P_t(\eta,\dd \xi) \Lambda_n(\xi,B) = \int \Lambda_n(\eta,\dd x) p_t^{[n]}(x,B)
\] 
for all $\eta\in \mathbf N_{<\infty}$ and all permutation invariant sets $B\in \mathcal E^n$. Hence, the kernel $\Lambda_n(\eta,B) = J_n(\one_B,\eta)$ intertwines the semigroups $(P_t)$ and $(p_t^{[n]})$. The second rephrasing uses the semi-group $(P_t)$ only, which makes the ``self'' in self-intertwining spring to the eye. Set
\begin{align*}
    \mathcal{K}(f,\eta) :=  f(0) + \sum_{n=1}^\infty \frac{1}{n!} \int f(\delta_{x_1} + \ldots + \delta_{x_n}) \eta^{(n)}(\dd(x_1, \ldots, x_n))
\end{align*}
for measurable bounded $f:\mathbf{N}_{<\infty} \to \R$ and $\eta\in \mathbf N_{<\infty}$.
Note that the integral vanishes for $n > \eta(E)$ and $\mathcal{K}(f, \sd) \in \mathcal{G}$ for $f\in \mathcal{G}$.
The function $\mathcal{K}(f, \sd)$ is also known as $K$-transform of $f$ (cf. \cite{lenard73}) and by linearity, it follows from \eqref{eq:pt-to-ptn} and \eqref{eq: self intertwining Jr} that $\mathcal{K}$ intertwines $(P_t)_{t\geq 0}$ with itself, i.e., 
\begin{equation} \label{eq:self-J}
	P_t \mathcal{K}(f, \sd)(\eta) = \mathcal{K}(P_t f, \eta).
\end{equation} 
for $f \in \mathcal{G}, \eta \in \mathbf{N}_{<\infty}$. For free Kawasaki dynamics, which is a special case of independent particles, this result is in fact known (see \cite[Section 3.2]{kondratiev2009hydrodynamic}). 

In terms of expectations, the self-intertwining relation becomes 
\begin{align*}
    \E_{\eta} \sbra{ \int f(\delta_{x_1} + \ldots + \delta_{x_n}) \eta_t^{(n)}(\dd(x_1, \ldots, x_n))} = \int \E_{\delta_{x_1} + \ldots + \delta_{x_n}} \sbra{f(\eta_t)} \eta^{(n)}(\dd(x_1, \ldots, x_n))
\end{align*}
for measurable, bounded $f:\mathbf{N}_{n} \to \R$, $n \in \N_0$ and $t \geq 0$.

To conclude we note a corollary on the time-evolution of correlation functions and explain the relation with Proposition~\ref{gensdt}. 

\begin{corollary}
    \label{cor:correlation-functions}
	Under the assumptions of Theorem~\ref{theorem: classical self-intertwining}, the following holds true for every initial condition $\eta\in \mathbf N_{<\infty}$. 
	Let $\alpha_n^t(B):= \E_\eta[\eta_t^{(n)}(B)]$ be the $n$-th factorial moment measure of the process $(\eta_t)_{t\geq 0}$ started in $\eta$. Then 
	\[
		\alpha_n^t(B) = \int \alpha_n^0(\dd x) p_t^{[n]}(x,B)
	\] 
	for all $n\in \N$, $t\geq 0$, and permutation-invariant sets $B\in \mathcal E^n$. 
\end{corollary} 

Of course for deterministic initial condition $\eta$ the time-zero factorial moment measure is just $\alpha_n^0 = \eta^{(n)}$, but in the form given above the relation generalizes to random initial conditions. 

\begin{proof} 
	We have
	\begin{equation*}
		\alpha_n^t(B) =\E_\eta\Bigl[ J_n(\one_B,\eta_t)\Bigr] = J_n(p_t^{[n]}\one_B,\eta) = \int \eta^{(n)}(\dd x) (p_t^{[n]} \one_B)(x) = \int \alpha_n^{0}(\dd x) p_t^{[n]}(x,B).  \qedhere
	\end{equation*} 
\end{proof} 

A generalized version of Proposition~\ref{gensdt} is recovered under the additional condition that for some $\sigma$-finite measure $\lambda$ on $E$ and each $n\in \N$, there exists a measurable function $u_t^{[n]}:E\times E\to \R_+$ with $u_t^{[n]}(x,y) = u_t^{[n]}(y,x)$ on $E\times E$ and 
\begin{equation} \label{eq:strong-duality}
	p_t^{[n]}(x,B) = \int_B u_t^{[n]}(x,y) \lambda^{\otimes n}(\dd y)
\end{equation}
for all $t > 0$, $x \in E^n$, and permutation invariant set $B\in \mathcal E^n$. This assumption shares similarites with the notion of duality from probabilistic potential theory, see Blumenthal and Getoor \cite[Chapter VI]{blumenthal-getoor}; we emphasize that the latter notion of (self-)duality with respect to a measure is stronger than reversibility of the measure. The additional condition is satisfied for example by independent reversible diffusions. Corollary~\ref{cor:correlation-functions}, \eqref{eq:strong-duality}, and the symmetry of $u_t^{[n]}$ yield
\begin{equation} \label{eq:gen2}
	\E_\eta\bigl[ \eta_t^{(n)}(B)\bigr] = \int_B \Biggl( \int_{E^n} u_t^{[n]}(y,x) \eta^{(n)}(\dd x)\Biggr)  \lambda^{\otimes n}(\dd y).
\end{equation}
This relation generalizes Proposition~\ref{gensdt}.

\subsection{Generalized Orthogonal Polynomials}
\label{section: Self-intertwining and orthogonality relations}

In this section we generalize the orthogonal self-duality relation introduced in Section \ref{section: orthogonal self-duality RW} to the class of  Markov processes on $\mathbf{N}_{<\infty}$ satisfying Assumption \ref{assumption: consistency and conserve number particles}. More precisely, assuming that there exists a reversible measure $\rho$, we show another self-intertwining relation where the intertwiner satisfies an orthogonality relation with respect to this measure. The intertwiner is a so-called \textit{generalized orthogonal polynomial}, a well studied object in the infinite dimensional analysis literature (see, e.g., \cite{schoutens}, \cite{CalculusVariationsProcessesIndependentIncrements} and \cite{Lytvynov2003}). We thus start by constructing the generalized orthogonal polynomials, following closely \cite{Lytvynov2003}.

Let $\rho$ be a probability measure on $(\mathbf{N}_{<\infty}, \mathcal{N}_{<\infty})$. We use the shorthand $L^2(\rho) := L^2(\mathbf{N}_{<\infty}, \mathcal{N}_{<\infty}, \rho)$. Through the rest of the section we assume that all moments of the total number of particles are finite. 

\begin{assumption}
    \label{assumption moments exist}
    Assume $\int \eta(E)^n \rho(\dd \eta) < \infty$
    for all $n \in \N$.
\end{assumption}

Assumption~\ref{assumption moments exist} implies that every map $\eta\mapsto \eta^{\otimes n}(f_n)= \int f_n\dd \eta^{\otimes n}$, with $f_n:E^n\to \R$ a bounded measurable function, is in $L^2(\rho)$. 

Orthogonal polynomials in a single real variable can be constructed by an orthogonalization procedure. This definition extends to the infinite-dimensional setting: generalized orthogonal polynomals are defined by
taking an  orthogonal projection onto a proper subspace of generalized  polynomials, see \cite{Lytvynov2003} and references therein. We thus define the space $\mathcal P_n$ of generalized polynomials (with bounded coefficients) of degree less or equal than $n \in \N_0$ as the set of linear combinations of maps $\eta\mapsto \int f_k\dd \eta^{\otimes k}$, $k\leq n$, with bounded measurable $f_k:E^k\to \R_+$, with the convention $\eta^{\otimes 0}(f_0):=f_0 \in \R$. Thus the set $\mathcal P_0$ consists of the constant functions. We refer to the functions $f_k$ as coefficients. 

Assumption 2 guarantees that every polynomial is square-integrable, i.e., $\mathcal P_n$ is a subspace of $L^2(\rho)$. In general it is not closed, we write $\overline{\mathcal P_n}$ for its closure in $L^2(\rho)$. The linear space $\mathcal P_n$ and its closure have the same orthogonal complement $\mathcal P_n^\perp = \overline{\mathcal P_n}^\perp$ in $L^2(\rho)$. 

The next definition is equivalent to a definition from \cite[Section~5]{Lytvynov2003}.

\begin{definition}[Generalized orthogonal polynomials]
For $n\in \N$ and $f_n:E^n \to \R$ a bounded measurable function we define the associated generalized orthogonal polynomial as follows
\begin{align*}
   I_n(f_n, \sd) := \text{orthogonal projection of } (\eta \mapsto \eta^{\otimes n}(f_n)) \text{ onto } \overline{\mathcal P_{n-1}}^\perp.
\end{align*}
\end{definition}

\noindent Equivalently, 
\[
	I_n (f_n,\eta) = \eta^{\otimes n}(f_n) - Q(\eta)
\] 
with $Q \in \overline{\mathcal P_{n-1}}$ the orthogonal projection of $\eta\mapsto \eta^{\otimes n}(f_n)$ onto $\overline{\mathcal P_{n-1}}$. Notice that $I_n(f_n,\eta)$ is only defined up to $\rho$-null sets. 

\begin{remark}[Wick dots and multiple stochastic integrals]
In the literature (see, e.g., \cite[Section~5]{Lytvynov2003}) the generalized orthogonal polynomial $I_n(f_n, \eta)$ is often denoted by  $:\eta^{\otimes n}(f_n):$ (``Wick dots''). When $\rho$ is the distribution of a Poisson point process with intensity measure $\lambda$, the generalized orthogonal polynomial is given by a multiple stochastic integral with respect to the compensated Poisson measure $\eta - \lambda$ (see the references provided at the end of Section~\ref{section: IRW}), hence the notation $I_n(f_n,\eta)$. The notation has the advantage of being analogous to the one  used for the self-intertwiner $J_n$ in Section \ref{section: generalized falling factorial polynomials}, which is why we keep it.
\end{remark}

\begin{remark}[Orthogonality relation]
It follows from the definition that 
\begin{align*}
    \int I_n(f_n, \sd) I_m(g_m, \sd) \dx \rho = 0
\end{align*}
for $n \neq m$. Moreover  $f_n \mapsto I_n(f_n, \sd)$ extends to a unitary operator on the space of permutation invariant functions that are square integrable with respect to some measure $\lambda_n$ (see, e.g., \cite[Corollary 5.2]{Lytvynov2003} for further details). When $\rho$ is the distribution of a Poisson process with intensity measure $\lambda$, the measure $\lambda_n$ is the product $\lambda_n=\lambda^{\otimes n}$, but in general the measure $\lambda_n$ is more complicated. 
\end{remark}

\begin{remark} [Chaos decompositions and L{\'e}vy white noise] Generalized orthogonal polynomials appear naturally in the study of non-Gaussian white noise \cite{berezansky1996infinite-dimensional,berezansky2002pascal}, they are used to prove chaos decompositions. The relation between polynomial chaos and chaos decompositions in terms of multiple stochastic integrals with respect to \emph{power jump martingales} \cite{nualart-schoutens2000} is investigated in detail \cite{Lytvynov2003}. 
Chaos decompositions play a role in the study of L{\'e}vy white noise and stochastic differential equations driven by L{\'e}vy white noise \cite{dinunno-oksendal-proske2004,lokka-proske2006,meyerbrandis2008}. 
\end{remark}

We complement the definition of the generalized orthogonal polynomials by two propositions on their properties. The first proposition says that the orthogonal polynomials can also be obtained by an orthgonal projection of the generalized falling factorial polynomials $\eta\mapsto J_n(f_n,\eta)$ instead of $\eta \mapsto \eta^{\otimes}(f_n)$. This observation plays an important role in the proof of Theorem~\ref{theorem: intertwining projection}.

\begin{proposition}
\label{proposition: properties ort pol 1}
The following identities hold 
\begin{align}
    \label{equation: spanned by Jk}
    \mathcal{P}_n &= \set{\eta \mapsto \sum_{k=0}^n J_k(f_k, \eta) : f_k \in \mathcal{F}_k, k \in \set{0, \ldots, n}, n \in \N_0}, \\
        \label{equation: projection Jk} 
 I_n(f_n, \sd) &= \text{orthogonal projection of }  J_n(f_n, \sd) \text{ onto } \overline{\mathcal P_{n-1}}^\perp, \quad f_n \in \mathcal{F}_n. 
    \end{align}
\end{proposition} 

We note that \eqref{equation: spanned by Jk} is a direct consequence of the fact that $J_k(f_k, \sd)$ can be written as linear combination of integrals with respect to the product measure of degree $\leq k$ and vice versa, see \cite[Eq.~(3.1)-(3.3)]{finkelshtein2021stirling}.  We provide a complete proof of the above proposition in  Appendix \ref{appendix: properties ort pol}. 

The second proposition applies under an additional assumption of complete independence. 
A finite point process $\zeta$ is \emph{completely independent} (or \emph{completely orthogonal}) \cite{LastPenroseLecturesOnThePoissonProcess} if the counting variables $\zeta(A_1),\ldots, \zeta(A_m)$ associated with pairwise disjoint regions $A_1,\ldots, A_m\in \mathcal E$, $m\in \N$, are independent. Complete independence implies a factorization property of generalized orthogonal polynomials with disjointly supported coefficients. 

\begin{proposition}
    \label{proposition: properties ort pol 2}
	Suppose that $\rho$ is the distribution of some finite completely independent point process. 
	Let $N \geq 2$, $A_1,\ldots, A_N\in \mathcal E$ pairwise disjont, and $d_1,\ldots, d_N\in\N_0$. Further let $f_i:E^{d_i}\to \R$, $i=1,\ldots,N$ be bounded measurable functions that vanish on $E^{d_i}\setminus A_i^{d_i}$. Set $n:=d_1+\cdots + d_N$. Then 
	\begin{align}
        \label{equation: factorization of orthogonal polynomials}
        I_n(f_1 \otimes \ldots \otimes f_n, \eta) = I_{d_1}(f_1, \eta) \cdots I_{d_n}(f_n, \eta)
    \end{align}
    for $\rho$-almost all $\eta\in \mathbf N_{<\infty}$. 	
\end{proposition}

The proposition is proven in Appendix~\ref{section: appendix proof of factorization}. For special cases of measures $\rho$ that give rise to orthogonal polynomials of Meixner's type, a similar factorization property is found, for example, in \cite[Lemma 3.1]{lytvynov2003JFA}. Our proposition instead holds true for all distributions of completely independent point processes.

\begin{remark} A particularly relevant case is when $f_i$ is the indicator of $A_i^{d_i}$. Then Proposition~\ref{proposition: properties ort pol 2} says that the orthogonalized version of $\eta\mapsto \prod_{i=1}^n \eta(A_i)^{d_i}$ is equal to the product of the orthogonalized versions of $\eta\mapsto \eta(A_i)^{d_i}$.
When $\rho$ is the distribution of a Poisson or Pascal point process (see Sections~\ref{section: examples} and~\ref{section: gSIP} below),  the orthogonalized version of $\eta(A_i)^{d_i}$ is in fact a univariate orthogonal polynomial in the variable $\eta(A_i) \in\N_0$ and we obtain a product of univariate orthogonal polynomials, see \eqref{eq:charlier-product} and~\eqref{eq: Meixner polynomials}. In general, however, the orthogonalized version of $\eta(A_i)^{d_i}$ need not be a univariate polynomial. 
\end{remark}

We now state the second main theorem of this section, which is the analogue of Theorem \ref{theorem: classical self-intertwining} but where the self-intertwiner is the generalized orthogonal polynomial introduced above.

\begin{theorem}[Self-intertwining relation]
    \label{theorem: intertwining projection}
    Let $(\eta_t)_{t\geq 0}$ be a Markov process on $\mathbf{N}_{<\infty}$ that satisfies Assumption~\ref{assumption: consistency and conserve number particles}, i.e.\ it is  consistent and conservative. Let $\rho$ be a reversible probability measure for $(\eta_t)_{t\geq 0}$ that satisfies Assumption~\ref{assumption moments exist}. Then, 
    \begin{align}\label{eq: intertwiner 2}
        P_t I_n(f_n, \cdot)(\eta) = I_n(p_t^{[n]}f_n, \eta)
    \end{align}
    for $\rho$-almost all $\eta\in \mathbf N_{<\infty}$, all $t \geq 0$, and all $f_n \in \mathcal{F}_n$.
\end{theorem}

\begin{proof} 
    To lighten notation, we drop the second variable $I_n(f_n,\sd)$ and write $I_n(f_n)$ when we refer to the function in $L^2(\rho)$, similarly for $J_n(f_n)$. 
    Let $\Pi_{n-1}$ be the orthogonal projection in $L^2(\rho)$ onto $ {\overline{\mathcal P_{n-1}}}$, and $\mathrm{id}$ the identity operator in $L^2(\rho)$. By Proposition~\ref{proposition: properties ort pol 1}, 
    \[
    	I_n(f_n) = (\mathrm{id} - \Pi_{n-1}) J_n(f_n).
    \] 
    The theorem follows once we know that the semigroup $P_t$ commutes with the projection $\Pi_{n-1}$ i.e.
    \begin{equation}\label{eq: commutation Semigroup and projection}
    	P_t \Pi_{n-1} = \Pi_{n-1} P_t
    \end{equation}
    since then, \eqref{eq: intertwiner 2} is   obtained as follows
    \begin{equation*}
        P_t I_n(f_n) = P_t (\mathrm{id} - \Pi_{n-1}) J_n(f_n) = (\mathrm{id} - \Pi_{n-1}) P_t J_n(f_n) = 
    		(\mathrm{id} - \Pi_{n-1}) J_n(p_t^{[n]}f_n ) = I_n(p_t^{[n]} f_n)
    \end{equation*}
    where we used Proposition~\ref{proposition: properties ort pol 1} in the first and the fourth equality and Theorem~\ref{theorem: classical self-intertwining} in the third equality.
    
    Let $k \leq n-1$ and let us recall the characterization of $\mathcal P_n$ given in Proposition~\ref{proposition: properties ort pol 1}. Using Theorem \ref{theorem: classical self-intertwining} combined with the fact that  $p_t^{[k]} f_k \in \mathcal{F}_k$ for all $f_k\in \mathcal{F}_k$, we have that $P_t J_k(f_k, \sd) = J_k (p_t^{[k]} f_k, \sd) \in \mathcal{P}_{n-1}$. Thus, for all $t\ge 0$ and $n \in \N_0$, 
    $P_t \mathcal{P}_{n-1} \subset \mathcal{P}_n$ and by the boundedness of $P_t$ on $L^2(\rho)$ we obtain \begin{equation}\label{eq: Pt and Pn}
    	P_t \overline{\mathcal{P}_{n-1}} \subset \overline{\mathcal{P}_{n-1}}.
    \end{equation}
    The operator $P_t$ is self-adjoint in $L^2(\rho)$ because of the reversibility of $\rho$. It is a general fact  that a bounded self-adjoint operator that leaves a closed vector space invariant commutes with the orthogonal projection onto that space.  Let us check this fact for our concrete operators and spaces. For $f \in \overline{\mathcal{P}_{n-1}}^\perp$, by  the self-adjointness of $P_t$ on $L^2(\rho)$ and \eqref{eq: Pt and Pn}, we  have, for all $g\in \overline{\mathcal{P}_n}$, that 
    $\int (P_t f) g\dd \rho=\int f(P_t g) \dd \rho=0$ and thus
    \begin{equation}\label{eq: Pt perp Pn}
        P_t \overline{\mathcal{P}_{n-1}}^\perp \subset \overline{\mathcal{P}_{n-1}}^\perp.
    \end{equation}
    We then have, using \eqref{eq: Pt and Pn}, \eqref{eq: Pt perp Pn} and $f - \Pi_{n-1} f\in  \overline{\mathcal{P}_{n-1}}^\perp$ that, for all $f\in L^2(\rho)$,
    \begin{align*}
        \Pi_{n-1}P_t f &= \Pi_{n-1}P_t\Pi_{n-1}f + \Pi_{n-1}P_t (f-  \Pi_{n-1} f) = P_t \Pi_{n-1} f.
    \end{align*}
    This completes the proof of~\eqref{eq: commutation Semigroup and projection} and the proof of the theorem.
\end{proof} 

\section{Examples}
\label{section: examples}

In this section we provide some examples of known consistent and conservative Markov processes, i.e. of processes satisfying Assumption \ref{assumption: consistency and conserve number particles}. Moreover, we also provide the reversible distribution of those processes, when known, and we specify when the assumptions of Theorem \ref{theorem: intertwining projection} are also satisfied. In particular, we recover known self-duality functions of systems of particles hopping on a finite set. In the next section, we introduce a new process, which generalizes the inclusion process (see, e.g., \cite{gkrv} where the SIP is introduced) for which both main theorems apply.

Before doing that, we recall the definition of the Charlier and Meixner polynomials, see e.g. \cite{HypergeometricOrthogonalPolynomials}, which are polynomials orthogonal with respect to the Poisson and negative binomial distribution. Differently from the usual definition in the literature, we normalize orthogonal polynomials to be monic where a polynomial $p(x) = a_0 + a_1 x + \ldots + a_n x^n$ is called monic if $a_n = 1$. These sequences of orthogonal polynomials can be expressed by using the generalized hypergeometric function given by
\begin{align*}
    \hyp{p}{q}{a_1, \ldots, a_p}{b_1, \ldots, b_q}{z} := \sum_{k=0}^\infty \frac{(a_1)^{(k)} \cdots (a_p)^{(k)}}{(b_1)^{(k)} \cdots (b_q)^{(k)}} \frac{z^k}{k!}
\end{align*}
for $a_1, \ldots, a_p, b_1, \ldots, b_q,z \in \R$, $p,q \in \N$, where we remind the reader that $(a)^{(0)} := 1$ and $(a)^{(k)} := a (a+1) \cdots (a+k-1)$ denotes the rising factorial (also called Pochhammer symbol). Similarly, we recall the falling factorial defined by $(a)_k := a (a-1) \cdots (a-k+1)$, $(a)_0 := 1$.

\begin{enumerate}[\normalfont(i)]
    \item The monic Charlier polynomials are given by
    \begin{align*}
        \mathscr{C}_{n}(x; \alpha) := (-\alpha)^n\ \hyp{2}{0}{-n, -x}{-}{-\frac{1}{\alpha}} = \sum_{k=0}^n \binom{n}{k} (-\alpha)^{n-k} (x)_k, \quad x \in \N_0
    \end{align*}
    for $n \in \N_0$ and $\alpha > 0$ and they satisfy the orthogonality relation
    $$
        \sum_{\ell=0}^\infty \mathscr{C}_{n}(\ell; \alpha) \mathscr{C}_{m}(\ell; \alpha) \Poi(\alpha)(\set{\ell}) = \one_{\{n=m\}} \alpha^n n!
    $$
    for $n, m \in \N_0$, i.e., $\mathscr{C}_{n}(\sd; \alpha)$ are orthogonal polynomials with respect to the Poisson distribution $\Poi(\alpha)(\set{\ell}) = e^{-\alpha} \frac{\alpha^\ell}{\ell!}$. $\ell \in \N_0$. 
    \item The monic Meixner polynomials are given by 
    \begin{multline*}
    	\mathscr M_{n}(x; a; p) := (a)^{(n)} \bra{1-\frac{1}{p}}^{-n} \hyp{2}{1}{-x, -n}{a}{1-\frac{1}{p}} \\
    	= \sum_{k=0}^n \binom{n}{k} \bra{1-\frac{1}{p}}^{k-n} (a+k)^{(n-k)} (x)_{k}, \quad x \in \N_0
    \end{multline*}
    for $n \in \N_0$, $a > 0$, $p \in (0, 1)$ and they satisfy the orthogonality relation
    \begin{equation}\label{eq:mortho-univariate}
        \sum_{\ell = 0}^\infty \mathscr M_{n}(\ell; a; p) \mathscr M_{m}(\ell; a; p) \NegBin(a, p)(\set{\ell}) = \one_{\{n=m\}} \frac{p^n n! (a)^{(n)}}{(1-p)^{2n}} 
    \end{equation}
    for $n, m \in \N_0$, i.e., $(\mathscr M_{n}(\sd; a; p))_{n \in \N_0}$ are orthogonal polynomials with respect to the generalized negative binomial distribution 
    \[
    	\NegBin(a, p)(\set{\ell}) = (a)^{(\ell)} \frac{p^\ell }{\ell!} (1-p)^{a},\quad \ell \in \N_0.
    \] 
\end{enumerate}

\subsection{Interacting Particle Systems on a Finite Set}
\label{section: example IPS on finite set}

Let $E$ be a non-empty finite set and identify $\xi \in \mathbf{N}_{<\infty}$ with $(\xi_k)_{k\in E} := (\xi(\set{x}))_{x \in E} \in \N_0^{E}$. Let $\eta$ be a Markov process on $\mathbf{N}_{<\infty}$ satisfying Assumption~\ref{assumption: consistency and conserve number particles} and $\rho$ be a reversible probability measure satisfying Assumption~\ref{assumption moments exist}. We then have that
$D^{\mathrm{cheap}}(\xi, \eta) := \frac{\one_{\{\eta = \xi\}}}{\rho(\set{\xi})}$, for $\eta, \xi \in \mathbf{N}_{<\infty}$ is the so-called cheap or trivial self-duality function (\cite[Eq.~(4.2)]{carinci2021consistent}). In this section we recover well-known self-duality functions of systems of particles hopping on a finite set by applying the intertwiners $J_n$ and $I_n$ to the cheap duality function. 
Note that 
\begin{align*}
    D^{\mathrm{cheap}}_n(\xi, x) := D^{\mathrm{cheap}}\bra{\xi, \sum_{k=1}^n  \delta_{x_k}}, \quad \xi \in \mathbf{N}_{<\infty}, x = (x_1, \ldots, x_n) \in E^n, n \in \N,
\end{align*}
is a duality functions for $(P_t)_{t\geq 0}$ and the $n$-particle semigroup $(p_t^{[n]})_{t\geq 0}$, i.e., $P_t D^{\mathrm{cheap}}_n(\sd, x)(\xi) = p_t^{[n]} D^{\mathrm{cheap}}_n(\xi, \sd)(x)$ for each $\xi \in \N_0^E$, $x \in E^n$, $n \in \N$. Putting $D^{\mathrm{cheap}}_0(\xi,) := D^{\mathrm{cheap}}\bra{\xi,0}$ yields $P_t D^{\mathrm{cheap}}_0(\sd,)(\xi) = p_t^{[0]} D^{\mathrm{cheap}}_0(\xi,)$.

It is well-known that applying an intertwiner to a duality function, for instance $D^{\mathrm{cheap}}_n(\xi, x)$, yields again a self-duality function, see e.g. \cite[Theorem 2.5]{carinci2019orthogonal}
or  \cite[Remark 2.7]{gkrv}.

\begin{proposition}
    Let $\rho = \bigotimes_{k\in E} \rho_k$ where $\rho_k$ are probability measures on $\N_0$ satisfying $\rho_k(\set{\ell}) > 0$ for each $\ell \in \N_0$. Consider for each $\rho_k$ the sequence of monic orthogonal polynomials $(\mathscr{P}_n(\sd, \rho_k))_{n \in \N_0}$. Then, 
    \begin{enumerate}
        \item 
        applying $J_n$ to $D^{\mathrm{cheap}}_n$ yields
        \begin{align*}
            \mathfrak{D}^{\mathrm{cl}}_n(\xi, \eta) := \frac{1}{n!} J_n( D^{\mathrm{cheap}}_n(\xi, \sd), \eta) = \one_{\{\xi(E) = n\}}\prod_{x\in E} \frac{1}{\rho_x(\set{\xi_x})\xi_x!} (\eta_x)_{\xi_x}, \quad n \in \N_0, \xi, \eta \in \mathbf{N}_{<\infty};
        \end{align*}
        \item applying $I_n$ to $D^{\mathrm{cheap}}_n$ yields
        \begin{align*}
            \mathfrak{D}^{\mathrm{ort}}_n(\xi, \eta) :=\frac{1}{n!} I_n( D^{\mathrm{cheap}}_n(\xi, \sd), \eta) = \one_{\{\xi(E) = n\}} \prod_{x \in E} \frac{1}{\rho_x(\set{\xi_x})\xi_x!} \mathscr{P}_{\xi_x}(\eta_x, \rho_x) , \quad n \in \N_0, \xi, \eta \in \mathbf{N}_{<\infty}.
        \end{align*}
    \end{enumerate}
\end{proposition}

As a consequence, Theorem~\ref{theorem: classical self-intertwining} and Theorem~\ref{theorem: intertwining projection} yield that $ \mathfrak D^{\mathrm{cl}}_n$ and $  \mathfrak D^{\mathrm{ort}}_n$ satisfy \eqref{equation: selfdual}, i.e., they are self-duality functions for $(P_t)_{t\geq 0}$ for each $n \geq \N_0$. Moreover, summing over $n$ in $ \mathfrak D^{\mathrm{cl}}_n$ and $  \mathfrak D^{\mathrm{ort}}_n$, we obtain the self-duality functions
\begin{align}
    \label{equation: classical duality function}
    \mathfrak{D}^{\mathrm{cl}}(\xi, \eta) := &\prod_{x\in E} \frac{1}{\rho_x(\set{\xi_x})\xi_x!} (\eta_x)_{\xi_x}, \quad \xi, \eta \in \mathbf{N}_{<\infty}, \\
    \label{equation: orthogonal duality function}
    \mathfrak{D}^{\mathrm{ort}}(\xi, \eta) := &\prod_{x \in E} \frac{1}{\rho_x(\set{\xi_x})\xi_x!} \mathscr{P}_{\xi_x}(\eta_x, \rho_x), \quad \xi, \eta \in \mathbf{N}_{<\infty}.
\end{align}

\begin{proof}
    Without loss of generality, let $E = \set{1, \ldots, N}$ and fix $\xi \in \N_0^N$, $n \in \N$. 
    \begin{enumerate}
        \item
        Note that 
        \begin{align}
            \label{equation: symmetrization indicator}
            \one_{\set{\xi = \delta_{x_1} + \ldots + \delta_{x_n}}} = \one_{\{\xi(E) = n\}} \frac{n!}{\xi_1! \cdots \xi_N !} \widetilde{ \one_{\set{1}}^{\otimes \xi_1} \otimes \cdots \otimes \one_{\set{N}}^{\otimes \xi_N}}(x_1, \ldots, x_n), \quad x_1, \ldots, x_n \in E
        \end{align}
        where $\widetilde{ \one_{\set{1}}^{\otimes \xi_1} \otimes \cdots \otimes \one_{\set{N}}^{\otimes \xi_N}}$ denotes the symmetrization of $\one_{\set{1}}^{\otimes \xi_1} \otimes \cdots \otimes \one_{\set{N}}^{\otimes \xi_N}$. Hence, using \eqref{eq: relation factorial from tensor products}, we obtain
        \begin{align*}
            \frac{1}{n!} J_n( D^{\mathrm{cheap}}_n(\xi, \sd), \eta) &= \frac{\one_{\{\xi(E) = n\}}}{\rho(\set{\xi}) \xi_1! \cdots \xi_N !} \int \one_{\set{1}}^{\otimes \xi_1} \otimes \cdots \otimes \one_{\set{N}}^{\otimes \xi_N} \dx \eta^{(n)} \\
            &= \one_{\{\xi(E) = n \}} \prod_{x=1}^N \frac{1}{\rho_x(\set{\xi_x}) \xi_x!}  (\eta_x)_{\xi_x}.
        \end{align*}
        for each $\xi \in \N_0^N$. 
    
        \item Let $\mathbf P_n := \overline{\mathcal P_n}\cap \overline{\mathcal P_{n-1}}^\perp$.
        By the orthogonal decomposition
        \begin{align*}
            \mathbf{P}_n = \bigoplus_{d_1 + \ldots + d_N = n } \mathrm{span} \set{\mathscr{P}_{d_1}(\sd, \rho_1) \otimes \cdots \otimes \mathscr{P}_{d_N}(\sd, \rho_N)}
        \end{align*}
        we obtain that the projection of $\N_0^N \ni \eta \mapsto \int \one_{\set{1}}^{\otimes \xi_1} \cdots \one_{\set{N}}^{\otimes \xi_N} \dx \eta^{\otimes n} = \eta_1^{\xi_1} \cdots \eta_N^{\xi_N}$ onto $\mathbf{P}_n$ is equal to $\eta \mapsto \mathscr{P}_{\xi_1}(\eta_1, \rho_1) \cdots \mathscr{P}_{\xi_N}(\eta_1, \rho_N)$. Therefore, using \eqref{equation: symmetrization indicator}
        \begin{align*}
            \frac{1}{n!} I_n( D^{\mathrm{cheap}}_n(\xi, \sd), \eta) &= \frac{\one_{\{\xi(E) = n\}}}{\rho(\set{\xi})\xi_1! \cdots \xi_N!} I_n(\one_{\set{1}}^{\otimes \xi_1} \otimes \cdots \otimes \one_{\set{N}}^{\otimes \xi_N}, \eta) \\
            &= \one_{\{\xi(E) = n\}} \prod_{x =1}^N \frac{1}{\rho_x(\set{\xi_x})\xi_x!} \mathscr{P}_{\xi_x}(\eta_x, \rho_x)
        \end{align*}
        for each $\eta \in \N_0^N$. 
    \end{enumerate}
\end{proof}

We consider three prominent examples of consistent and conservative Markov processes on $\N_0^E$. For a characterization of consistent particle system on countable $E$ we refer to %\cite[Theorem 3.2]{carinci2019consistent} 
\cite[Theorem 3.3]{carinci2021consistent}. Let $\abs{E} \geq 2$, $c=\{c_{\{x,y\}}, x,y \in E \}$ be a set of symmetric and non-negative conductances, such that $(E,c)$ is connected and let $(\alpha_y)_{y \in E} \subset \N$. Then, for $\sigma \in \set{-1, 0, 1}$, the Markov process with infinitesimal generator acting on functions $f : \mathbf{N}_{<\infty} \to \R$ as
\begin{align*}
    Lf(\eta) &= \sum_{x,y \in E} c_{\set{x,y}} \bra{f(\eta - \delta_{y} + \delta_x) - f(\eta) } ( \alpha_y + \sigma \eta(\set{y})) \eta(\set{x}), \quad \eta \in \mathbf{N}_{<\infty}
\end{align*}
is a consistent and conservative process. In particular, for $\sigma = -1$, we obtain the inhomogeneous partial exclusion process (SEP) (see, e.g., \cite[Eq.~(1.3)]{floreani2020hydrodynamics}), for $\sigma=0$ a system of independent random walkers (IRW) and for $\sigma=1$ the inhomogeneous inclusion process SIP (see, e.g., \cite[Eq.~(2.2)]{floreani2020orthogonal}). 
  
By a simple detailed balance computation one can show that, for those processes, there exists a one parameter family $\{\rho_\theta, \ \theta\in \Theta\}$ with $\Theta=(0,1]$ for $\sigma=-1$ and $\Theta=(0,\infty)$ for $\sigma\in\{0,1\}$ of reversible measures, namely (cf. \cite[Eq.~(3.1)]{floreani2020orthogonal}) $\rho_\theta:=\bigotimes_{x\in E} \rho_{x,\theta}$ with
\begin{align*}
    \rho_{x,\theta} = \begin{cases} \Bin(\alpha_x,\theta) &\text{if } \sigma = -1\\[.15cm]
    \Poi(\alpha_x \theta) &\text{if } \sigma = 0\\[.15cm]
    \NegBin\bra{\alpha_x, \frac{\theta}{1+\theta}} &\text{if } \sigma = 1\ .
    \end{cases}
\end{align*}
Using that $ \rho_{x,\theta}(\set{n}) = \frac{w_x(n)}{z_{x,\theta}}\, \bra{\frac{\theta}{1+\sigma \theta} }^{n}  \frac{1}{n!}$ where
\begin{align*}
    w_x(n) := \begin{cases}
     (\alpha_x)_{n} &\text{if\ } \sigma = -1\\
     \alpha_x^n &\text{if } \sigma = 0\\
     (\alpha_x)^{(n)} &\text{if } \sigma = 1\
     \end{cases}
    \text{ and }
     z_{x,\theta} := \begin{cases}
     (1-\theta)^{-\alpha_x}&\text{if } \sigma = -1\\
     e^{\alpha_x \theta} &\text{if } \sigma = 0\\
     (1+\theta)^{\alpha_x} &\text{if } \sigma = 1\ 
     \end{cases}
\end{align*}
in \eqref{equation: classical duality function} we obtain 
\begin{align*}
    \mathfrak D^{\mathrm{cl}}(\xi, \eta) 
    &= \left(\frac{\theta}{1+\sigma \theta} \right)^{-\xi( E ) } \left(\prod_{x\in E} z_{x,\theta}\right) \prod_{x\in E} \frac{(\eta_x)_{\xi_x}}{w_x(\xi_x)}
\end{align*}
which are the classical duality functions for $(\eta_t)_{t \geq 0}$ (see, e.g., \cite[Eq.~(2.16)]{floreani2020orthogonal}). Notice that, due to Assumption~\ref{assumption: consistency and conserve number particles}~(\ref{item: conserving number particles}), the term $\bra{\frac{\theta}{1 + \sigma \theta} }^{-\xi( E ) } \bra{\prod_{k\in E} z_{k,\theta}}$ is constant in time and, thus, it does not play any role in the duality relation. 

For these systems, the self-dualities provided by \eqref{equation: orthogonal duality function} coincide (up to a multiplicative constant depending on the total number of particles which is a conserved quantity) to the orthogonal dualities studied in \cite{redig_factorized_2018}, \cite{franceschini2019stochastic} and \cite{floreani2020orthogonal} which are given by product of Charlier polynomials for $\sigma=0$, products of Meixner polynomials for $\sigma=-1$ and products of Krawtchouk polynomials (see, e.g., \cite[Eq.~(9.11.1)]{HypergeometricOrthogonalPolynomials}) for $\sigma=-1$. Indeed, considering, for instance, the system of independent random walks, the self-duality function of \eqref{equation: orthogonal duality function} turns into
\begin{align*}
    \mathfrak{D}^{\mathrm{ort}} (\xi, \eta) &= \prod_{k\in E} \frac{1}{\rho_k(\set{\xi_k}) \xi_k!} \mathscr{C}_{\xi_k}(\eta_k, \alpha_k) \\
    &= \prod_{k\in E} \frac{1}{e^{-\alpha_k} \alpha_k^{\xi_k}} (-\alpha_k)^{\xi_k} \hyp{2}{0}{-\xi_k, -\eta_k}{-}{-\frac{1}{\alpha_k}} \\
    &= e^{\alpha(E)} (-1)^{\xi(E)} \prod_{k\in E}  \hyp{2}{0}{-\xi_k, -\eta_k}{-}{-\frac{1}{\alpha_k}} 
\end{align*}
coinciding with the orthogonal self-dualities given in literature mentioned above. The same holds also for the exclusion and the inclusion process.

\subsection{Independent Markov Processes}

Every system of independent Markov processes (e.g.\ the free Kawasaki dynamics \cite{kondratiev2009hydrodynamic}, independent Brownian motions) is consistent and conservative. 
For independent particles, our theorems results allow us to recover known results on intertwining with Lenard's $K$-transform and multiple stochastic integrals, see \cite{kondratiev2009hydrodynamic, surgailis84} and the references therein. Our contribution is the proof that these intertwining relations correspond exactly to classical and orthogonal dualities for independent random walkers on lattices from \cite[Proposition~2.9.4]{DeMasiPresutti} and \cite[Theorem~4]{franceschini2019stochastic}.

Let $(p_t)_{t\geq 0}$ be a Markov transition function on $(E,\mathcal E)$. The transition function for $n$ independent labelled particles with one-particle evolution governed by $(p_t)_{t\geq 0}$ has transition function $p_t^{\otimes n}$ uniquely determined by
\[
	p_t^{\otimes n}\bigl( x_1,\ldots, x_n;A_1\times\cdots \times A_n) = \prod_{i=1}^n p_t(x_i; A_i)\quad x_1,\ldots, x_n\in E,\ A_1,\ldots, A_n\in \mathcal E. 
\] 
The family of transition functions $(p_t^{\otimes n})_{t\geq 0}$, $n\in \N$ is strongly consistent and therefore the associated conservative transition function $(P_t)_{t\geq 0}$ (see \eqref{eq:forgetting-labels}) is consistent.

 Hence,  Theorem~\ref{theorem: classical self-intertwining} applied to the process  
$(\eta_t)_{t\geq 0}$ with transition function $(P_t)_{t\geq 0}$ yields the self-intertwining relation 
  $P_t J_n(f_n, \sd)(\eta) = J_n(p_t^{[n]} f_n, \eta)$ or more concretely, 
\[
	\E_{\eta} \Bigl[ \int f_n \dd \eta_t^{(n)}\Bigr] = \int (p_t^{\otimes n} f_n) \dd \eta. 
\] 
The relation holds true for all $t\geq 0$, all initial values $\eta\in \mathbf N_{<\infty}$, and all $f_n\in \mathcal F_n$. As noted in~\eqref{eq:self-J}, it implies that Lenard's $K$-transform and the semigroup $(P_t)_{t\geq 0}$ commute. Hence, for free Kawasaki dynamics, we recover a  relation from  \cite[Section 3.2]{kondratiev2009hydrodynamic}.

If we find a $\sigma$-finite reversible measure $\lambda$ for the one-particle dynamics $(p_t)_{t\geq 0}$, then the distribution of a Poisson process with intensity measure $\lambda$, denoted by $\pi_\lambda$, is reversible for $(\eta_t)_{t\geq 0}$. This property is a version of Doob's Theorem (cf.~\cite[Theorem 2.9.5]{DeMasiPresutti}) and of the displacement theorem (cf.~\cite{kingman1992poisson}). Moreover, $\lambda^{\otimes n}$ is reversible for $(p_t^{\otimes n})_{t\geq 0}$. For finite $\lambda$, the assumptions of Theorem \ref{theorem: intertwining projection} are satisfied and the self-intertwining relation $P_t I_n(f_n, \sd)(\eta) = I_n(p_t^{\otimes n} f_n, \eta)$ holds for $\pi_\lambda$-almost all $\eta \in \mathbf{N}_{<\infty}$, all $f_n \in \mathcal{F}_n$ and all $t \geq 0$. 

The construction of the generalized orthogonal polynomial with respect to the Poisson point process is standard and it is well-known that the orthogonality relation
\begin{align*}
    \int I_n(f_n, \sd) I_m(g_m, \sd) \dx \pi_\lambda = \one_{\set{n = m}} n! \int f_n g_m \dd \lambda^{\otimes n}
\end{align*}
holds for bounded $f_n \in \mathcal{F}_n, g_m \in \mathcal{F}_m$, $n, m \in \N_0$, with $\int f_0 g_0 \dd \lambda^{\otimes 0} := f_0 g_0$, and they generalize the Charlier polynomial in the following sense (see, e.g., \cite[Eq.~(3.3)]{Last2011}, 
\begin{equation} \label{eq:charlier-product}
    I_n(\one_{B_1}^{\otimes d_1} \otimes \cdots \otimes \one_{B_N}^{\otimes d_N}, \eta) = \prod_{k=1}^N \mathscr{C}_{d_k}(\eta(B_k); \lambda(B_k))
\end{equation}
for $\pi_\lambda$-almost all $\eta\in \mathbf N_{<\infty}$, $d_1 + \ldots + d_N = n$ and all pairwise disjoint $B_1, \ldots, B_N \in \mathcal{E}$. Yet another viewpoint is that $I_n(f_n,\sd)$ are multiple stochastic integrals with respect to the compensated Poisson measure $\eta-\lambda$. The reader interested in the relation between the generalized orthogonal polynomials $I_n(f_n,\sd)$ and multiple Wiener-Itô integrals, chaos decompositions, and Fock spaces is referred to \cite{Last2016}, \cite{meyer2006quantum} and \cite{Lytvynov2003}. 

In the language of multiple stochastic integrals, the intertwining  relation from Theorem~\ref{theorem: intertwining projection} says that applying the semigroup to the $n$-fold integral of $f_n$ is the same as the $n$-fold integral of $p_t^{\otimes n} f_n$.

\subsection{The Howitt--Warren Flow and a Consistent Family of Sticky Brownian Motions}

\label{section:ex-strongly-consistent}

As noted in Section~\ref{section:setting}, every strongly consistent family $(p_t^{[n]})_{t\geq 0}$, $n\in \N$, of transition functions induces a consistent semigroup $(P_t)_{t\geq 0}$. Strongly consistent families have been studied in the context of stochastic flows: Le Jan and Raimond \cite{LeJanRaimond} investigate a one-to-one correspondence between strong consistency families and stochastic flows of kernels.

A particular and well studied case is the Howitt--Warren flow. It is a stochastic flow of kernels whose 
$n$ point motions is given by a family of $n$ interacting Brownian motions that interact, roughly, by sticking together for a while when they meet. The interacting diffusions can be constructed, for example, as solutions to a  martingale problem \cite{howitt-warren2009}. Theorem~\ref{theorem: classical self-intertwining} applies to the semigroup $(P_t)_{t\geq 0}$ induced by the strongly consistent family of transitions functions $(p_t^{[n]})_{t\geq 0}$, $n\in \N$,  for $n$ sticky Brownian motions.

The dynamics of sticky Brownian motion depends on a choice of parameters and includes diffusions with a drift. For zero drift and a special choice of parameters, Brockington and Warren \cite{brockington2021bethe} prove, using a Bethe ansatz, an explicit formula for transition probabilities and the reversibility of the $n$-point motions with respect to some explicit measure $m_\theta^{(n)}$. They work on the Weyl chambers $\bar{W}^n:=\{x\in\R^n: x_1\ge \ldots \ge x_n\}$ and show that the transition function is of the form $p_t^{[n]}(x,\dd y) = u_t^{(n)}(x,y) m_\theta^{(n)}(\dd y)$ for some symmetric function $u_t^{(n)}(x,y)=u_t^{(n)}(y,x)$. With this the self-intertwining relation from Theorem~\ref{theorem: classical self-intertwining} can be rewritten as 
\begin{equation}
	\E_{\eta}\left[ \int_{\bar{W}^r} f_r( y) \eta^{(r)}(\dd y) \right]=\int_{\bar{W}^r} f_r( y)\left[\int_{\bar{W}^r} u^{(r)}_t(y, x) \eta^{(r)}(\dd x)\right]  m^{(r)}_\theta
(\dd y).
\end{equation}
Thus we obtain an identity analogous to \eqref{vivi} and~\eqref{eq:gen2}.

As the reversible measures $m_\theta^{(n)}$ from \cite{brockington2021bethe} have infinite total mass, it is not possible to construct from them a reversible measure supported on configurations of finitely many particles and Theorem~\ref{theorem: intertwining projection} on orthogonal intertwining relations is not applicable. 
We leave the study of the orthogonal self-intertwining relation for the system of sticky Brownian motions for future research.

For other examples of strongly consistent families, beyond sticky Brownian motions, we refer to \cite{LeJanRaimond} and \cite{schertzer_sun_swart_2016}. 

\section{Generalized Symmetric Inclusion Process}
\label{section: gSIP}
As an example of interacting system of particles jumping on a general Borel space $(E, \mathcal E)$, we introduce here a new process which is a natural extension of the SIP. Coherently with the setting of this paper, we consider the finite particle case only. Extension to the infinite particle case is not part of the scope of the present work and it is left for future research.

The SIP on countable sets was introduced in \cite{giardina2007duality} as a dual process of a model of heat conduction, which shares some features with the well-studied KMP model (see \cite{kipnis1982heat}). The process also appears, with a different interpretation, in mathematical population genetics. Indeed, in \cite[Section~5]{CARINCI2015941}, it is proved that the generator of the SIP coincide with the generator of an instance of the Moran model, which is dual to the Wright-Fisher diffusion process. Moreover, the scaling limit of the Moran model is the celebrated Fleming-Viot superprocess (see \cite{Etheridge00} and references therein) which has been studied using duality as well.

\subsection{Introducing the gSIP}

Let $\alpha$ be a finite, non-zero measure on $E$ and $c:E\times E\to \R_+$ a bounded symmetric function with $c(x,x)=0$ for all $x \in E$. The generalized symmetric inclusion process (gSIP) is the process with formal generator 
\begin{equation}\label{eq:sip-generator}
	L f(\eta) = \iint \bigl( f(\eta- \delta_x+\delta_y) - f(\eta)\bigr) c(x,y) (\alpha + \eta)(\dd y) \eta(\dd x). 
\end{equation} 
It is a continuous-time jump process with jump kernel 
\begin{equation} \label{eq:sip-jumpkernel}
	Q(\eta,B) = \iint \one_B(\eta- \delta_x+\delta_y) c(x,y) (\alpha + \eta)(\dd y) \eta(\dd x)
\end{equation} 
and it can be viewed, when $E=\R^d$, as a particular case of a  Kawasaki dynamics (see, e.g., \cite{KondratievLytvynov2007}).
Notice that $Q(\eta,E) <\infty$ for finite measures $\alpha$ and finite configurations $\eta\in \mathbf N_{<\infty}$. Accordingly the process $(\eta_t)_{t\geq 0}$ can be constructed with the usual jump-hold construction and the semigroup $(P_t)_{t\geq 0}$ is the minimal solution of the backward Kolmogorov equation, see Feller~\cite{Feller}.

The process is non-explosive because the particle number is conserved and $\sup\{Q(\eta,E): \eta(E) = n\}<\infty$ for every particle number $n\in \N_0$. Therefore the minimal solution $(P_t)_{t\geq 0}$ is a Markov semigroup ($P_t(\eta,E) = 1$ rather than $\leq 1$) and it is in fact the unique solution of the backward Kolmogorov equation. 

\begin{remark}
    \begin{enumerate}
        \item As we will see later, the gSIP $(\eta_t)_{t\geq 0}$ has the following connection to the well-known SIP of particles hopping on a finite set. 
        Let $A_1, \ldots, A_m \in \mathcal{E}$, $m \in \N$ be a partition of $E$ and let $c$ be constant on $A_i \times A_j$ and equal to $d_{ij}$ for each $i,j$. Then, the process $(\eta_t(A_1), \ldots, \eta_t(A_m))$ starting at $\eta_0 \in \mathbf{N}_{<\infty}$ behaves like a SIP on the finite set $\set{1, \ldots, m}$ with initial configuration $(\eta_0(A_1), \ldots, \eta_0(A_m))$ and transition rates $d_{ij}$. 
        \item Notice that a direct generalization of the Exclusion process analogous to the gSIP, would not be meaningful in general, because the probability to jump on already occupied points is zero whenever the jumping kernel of the single particle is not atomic. Thus an exclusion rule miming the one in the discrete setting cannot be modelled in the continuum. 
        
        \item The dynamics can be described informally as follows. Starting from an initial configuration $\eta_0 = \eta$ with $n= \eta(E)$ points $x_1,\ldots, x_n$, set 
    	\[
    		q_{i0}:= \int c(x_i,y) \alpha(\dd y),\quad q_{ij} := c(x_i,x_j),\quad z_i:= \sum_{j=0}^n q_{ij}, \quad z:= \sum_{i=1}^n z_i
    	\] 
        and do the following:
        \begin{enumerate}[\normalfont(i)]
        	\item Wait for an exponential time with parameter $Q(\eta,E)= z$. 
        	\item When time is up, choose one out of the $n$ points $x_1,\ldots, x_n$ of $\eta$. The point $x_i$ is chosen with probability $z_i/z$. Move the chosen point $x=x_i$ to a new location $y$: 
        	\begin{itemize}
        		\item With probability $q_{ij} /z_i$, the new location $y$ is equal to $y=x_j$.
        		\item With probability $q_{i0}/z_i$, the new location $y$ is chosen according 
        	to the probability measure $\alpha(E)^{-1} \alpha(\dd y)$.
        	\end{itemize}		
        \end{enumerate} 
        Then, repeat. The resulting process $(\eta_t)_{t\geq 0}$ and the associated semigroup $(P_t)_{t\geq 0}$, given by the minimal solution to the backward Kolmogorov equation, is in fact the unique solution. 
    \end{enumerate}
\end{remark}

\subsection{Reversibility and Intertwiners for the gSIP}

Fix $p\in (0,1)$. A \emph{Pascal point process} with parameters $p$ and $\alpha$ is a point process with the following properties:
\begin{enumerate}[\normalfont(i)]
	\item
	\label{item: complete independent}
	If $B_1,\ldots, B_m\in \mathcal E$ are disjoint then $\zeta(B_1),\ldots, \zeta(B_m)$ are independent. 
	\item 
	\label{item: negative binomial distribution}
	For every $B\in \mathcal E$, the distribution of $\zeta(B)$ is given by a negative binomial law:
	\begin{align*}
		\P\bigl(\zeta(B) = k\bigr) &= \bigl(1-p\bigr)^{\alpha(B)} \alpha(B) \bigl(\alpha(B)+1\bigr)\cdots \bigl(\alpha(B) + k-1\bigr)\, \frac{p^k}{k!}, \quad k\in \N_0. 
	\end{align*} 
	For $k=0$, the equation is to be read as $\P(\zeta(B) =0) = (1-p)^{\alpha(B)}$. 
\end{enumerate} 
The \emph{Pascal distribution} is the distribution of a Pascal point process and it is a direct generalization of the product measure of negative binomial distributions that is reversible for SIP. Indeed, the measure $\otimes_{x \in E} \NegBin(\alpha_x, p)$, $\alpha_x > 0$ can be seen as a Pascal distribution. Property \eqref{item: complete independent} follows immediately whereas \eqref{item: negative binomial distribution} follows from the fact that if $n_x\sim\NegBin(\alpha_x, p)$ and $n_y\sim \NegBin(\alpha_y, p)$, with $n_x$ and $n_y$ independent for $x\ne y\in E$, then $n_x+n_y\sim \NegBin(\alpha_x+\alpha_y, p)$.

\begin{theorem}
    \label{theorem: gSIP}
    Let $\alpha$ be a finite measure on $E$. Then
    \begin{enumerate}[\normalfont(i)]
        \item the generalized symmetric inclusion process with formal generator~\eqref{eq:sip-generator} is a consistent Markov process and thus the intertwining relation \eqref{eq: self intertwining Jr} with generalized falling factorials holds;
        \item for every $p \in (0, 1)$, the Pascal distribution $\rho$ with parameters $\alpha$ and $p$ is reversible and thus, the intertwining relation \eqref{eq: intertwiner 2} with generalized orthogonal polynomials holds. 
	\end{enumerate}
\end{theorem}

Notice that we have a family of reversible measures, indexed by $p\in (0,1)$, moreover the reversible Pascal distributions do not depend on the function $c(x,y)$ in the dynamics. 

Theorem~\ref{theorem: gSIP}(ii) is complemented by a concrete relation of the abstract generalized orthogonal polynomials $I_n(f_n,\cdot)$ with the univariate Meixner polynomials defined in Section~\ref{section: example IPS on finite set}. Generalized orthogonal polynomials of Meixner's type have been studied intensely in the context of non-Gaussian white noise \cite{berezansky1996infinite-dimensional,berezansky2002pascal}. Connections with quantum probability and representations of $^*$-Lie algebras and current algebras are investigated in \cite{accardi2002renormalized, accardi2009quantum}.

The following proposition is a variant of Lemma 3.1 in \cite{lytvynov2003JFA}. We give a self-contained proof in Section~\ref{sec:meixner-properties} that does not use the machinery of Jacobi fields or distribution theory.

\begin{proposition}
    \label{prop:in-to-meixner}
    The intertwiner $I_n$ is related to the Meixner polynomials via
    \begin{align}
        \label{eq: Meixner polynomials}
        I_n(\one_{B_1}^{\otimes d_1} \otimes \cdots \otimes \one_{B_N}^{\otimes d_N}, \eta) = \prod_{k=1}^N \mathscr M_{d_k}\bra{\eta(B_k); \alpha(B_k); p}.
    \end{align}
    for $\rho$-almost all $\eta \in \mathbf N_{<\infty}$ and all  pairwise disjoint $B_1, \ldots, B_N \in \mathcal{E}$, $n \in \N_0$, $d_1, \ldots, d_N$, $N \in \N$ with $d_1 + \ldots + d_N = n$. 
\end{proposition}

We define a measure $\lambda_n$ on $E^n$ that replaces the product measure $\lambda^{\otimes n}$ in the Poisson-Charlier case. Let $\Sigma_n$ be the collection of set partitions of $\{1,\ldots, n\}$. For $\sigma\in \Sigma_n$ and $g:E^n\to \R$, let $|\sigma|$ be the number of blocks of the partition $\sigma$ and $g_\sigma:E^{|\sigma|}\to \R$ the function obtained by identifying, in order of occurrence, those arguments that belong to the same block of $\sigma$. Define
\begin{equation}
    \label{eq:lambdandef}
	\lambda_n(B) = \sum_{\sigma \in \Sigma_n} \bra{\prod_{A\in \sigma} (\abs{A}-1)!} \int (\one_B)_{\sigma} \, \dd \alpha^{\otimes |\sigma|},\quad B\in \mathcal E^n.
\end{equation}
For example  $\lambda_1 = \alpha$ and 
\[
	\lambda_2(B) = \iint \one_B(x,y) \alpha(\dd x) \alpha (\dd y) + \int \one_B(x,x) \alpha(\dd x) 
\]
for all $B\in \mathcal E^2$.  Further set $\int f_0 g_0 \dx \lambda_0 := f_0 g_0$ for $f_0, g_0 \in \mathcal{F}_0=\R$.

The following proposition generalizes the univariate orthogonality relation~\eqref{eq:mortho-univariate}. It is similar to Corollary~5.2 in \cite{lytvynov2003JFA}, we provide a self-contained proof in Section~\ref{sec:meixner-properties}.

\begin{proposition}
    \label{prop:meixner-lambdan}
	The following orthogonality relations holds
    \begin{align}
        \label{eq: Orthogonality SIP}
    	\int I_n (f_n, \sd) I_m(g_m, \sd) \dx \rho = \one_{\{n = m\}} \frac{p^n n!}{(1-p)^{2n}} \int f_n g_m \dx \lambda_n
    \end{align}
    for $f_n \in \mathcal{F}_n$, $g_m \in \mathcal{F}_m$, $n,m \in \N_0$.
\end{proposition}

\begin{remark}[Sequential construction of $\lambda_n$]
    For $n\in \N$, define a kernel $k_{n,n+1}:E^n\times \mathcal E^{n+1}\to \R_+$ by
    \[
    	k_{n,n+1}(x_1,\ldots,x_n;B) = \int \one_B(x_1,\ldots,x_n,y) \alpha(\dd y) + \sum_{i=1}^n \one_B(x_1,\ldots,x_n,x_i). 
    \] 
    Then $\lambda_{n+1}= \lambda_n k_{n,n+1}$ meaning that $\lambda_{n+1}(B) =\int_{E^n} \lambda_n(\dd x) k_{n,n+1}(x,B)  $ for all $B\in \mathcal E^{n+1}$. Thus $\lambda_n$ is formed by adding points one by one; at each step, a new point either joins a pile of existing points or is placed at a new location $y$. This relation on the one hand connects to the very definition of the dynamics of the gSIP and on the other hand is reminiscent of the Chinese restaurant process used in sequential constructions for random partitions \cite[Chapter 3]{pitman2006combinatorial}. Notice that, upon normalization by the total mass of $\lambda_n$, \eqref{eq:lambdandef} gives rise to a probability measure on the set $\Sigma_n$ of partitions, related to the Ewens sampling formula. 
\end{remark}

\subsection{Proof of Theorem \ref{theorem: gSIP}}

Here we prove Theorem \ref{theorem: gSIP}. In addition, we remind the reader of an explicit description of the Pascal process as a compound Poisson process.

\paragraph{Consistency (Proof of Theorem~\ref{theorem: gSIP}(i))}
\label{section: consistency of gSIP}
We start by proving that the gSIP is consistent (see Definition \ref{def: consistent particle system}). Since we consider the finite particle case only, it is enough to check the commutation property in Definition \ref{def: consistent particle system} for the generator instead of the semigroup, i.e., $\mathcal{A} L f(\eta) = L\mathcal{A} f(\eta)$ for all $f \in \mathcal{G}$ and $\eta \in \mathbf{N}_{<\infty}$. Indeed, decompose the generator as $L=L_1+ L_2$ with
$$L_1f(\eta) := \iint \bra{f(\eta - \delta_x + \delta_y) - f(\eta)} c(x,y) \alpha(\dd y) \eta(\dd x)$$ and $$L_2f(\eta) := \iint \bra{f(\eta - \delta_x + \delta_y) - f(\eta)} c(x,y) \eta(\dd y) \eta(\dd x).$$ Notice that $L_1$ is the generator of a system of independent Markov processes, namely, independent random walkers with transition kernel given by $c(x,y) \alpha(\dd y)$. Thus, it is straightforward to check that $\mathcal{A} L_1 f(\eta) = L_1 \mathcal{A} f(\eta)$. It remains to show that 
\begin{equation}
    \label{eq: commutation L2 gsip}
    \mathcal{A} L_2 f(\eta) = L_2 \mathcal{A} f(\eta).
\end{equation}
First, we compute
\begin{align*}
    &L_2 \mathcal{A} f(\eta) \\&= \iiint f(\eta - \delta_x + \delta_y - \delta_z) \eta(\dd z) c(x,y) \eta(\dd y) \eta(\dd x) 
       - \iint f(\eta - 2\delta_x + \delta_y) c(x,y) \eta(\dd y) \eta(\dd x) \\
        &+\iint f(\eta - \delta_x) c(x,y) \eta(\dd y) \eta(\dd x) 
    -  \iiint f(\eta - \delta_z) \eta(\dd z) c(x,y) \eta(\dd y) \eta(\dd x) 
\end{align*}
second, 
\begin{align*}
    \mathcal{A} L_2 f(\eta)& = \iiint \bra{f(\eta - \delta_z - \delta_x + \delta_y) - f(\eta - \delta_z)} c(x,y) (\eta - \delta_z)(\dd y) (\eta - \delta_z)(\dd x) \eta(\dd z) \\
    &=L_2 \mathcal{A} f(\eta)\\&- \iint \bra{f(\eta - \delta_x ) - f(\eta - \delta_z)} c(x,z) \eta(\dd x) \eta(\dd z) + \int \bra{f(\eta - \delta_z) - f(\eta - \delta_z)} c(z,z) \eta(\dd z).
\end{align*}
Because the last two integrals above are both $0$, we obtain \eqref{eq: commutation L2 gsip} and the proof is concluded.
\qed

\paragraph{Explicit Representation of the Pascal Process.}

The Pascal process, also called negative binomial process, is a well-known point process (cf. \cite{serfozo1990point}, \cite{kozubowski2008distributional}). For the reader's convenience we recall the construction of that process. 

Fix $p \in (0,1)$ and a finite measure $\alpha$. Note that the Pascal point process has the structure of a measure-valued Lévy process, since $\zeta(A_1), \ldots, \zeta(A_n)$ are independent for pairwise disjoint $A_1, \ldots, A_n \in \mathcal{E}$ and the distribution of $\zeta(A)$ only depends on $\alpha(A)$, $A \in \mathcal{E}$. For more details, see \cite{CompletelyRandomMeasures,kingman1992poisson,Kallenberg2017}, 

More precisely, the Pascal process can be constructed in the following way, compound Poisson process (see \cite[Chapter 15]{LastPenroseLecturesOnThePoissonProcess}), i.e., 
\begin{align*}
    \zeta(A) := \int_{A \times \N} y\, \xi(\dd(x,y)), A \in \mathcal{E}
\end{align*}
where $\xi$ is a Poisson point process on $E \times \N$ with intensity measure $\lambda := \alpha \otimes \nu$ where the Lévy measure is given by
$
    \nu := \sum_{n=1}^\infty \frac{p^n}{n} \delta_n
$. It can readily checked that the Laplace functional is given by 
\begin{align}
    \label{Eq: Laplace functional Pascal process}
    L_\zeta(f) := \E \sbra{e^{-\int f \dx \zeta}} = \exp\bra{ \int \bra{e^{-yf(x)} - 1} \lambda(\dd(x,y)) } = \exp\bra{- \int \Phi(f(x)) \alpha (\dd x) }
\end{align}
with $\Phi(y) := \log \bra{ \frac{1-pe^{-y}}{1-p}}$, $y \geq 0$. Equation~\eqref{Eq: Laplace functional Pascal process} implies for $A \in \mathcal{E}$
\begin{align*}
    \E\bra{e^{-\zeta(A) s}} = \exp\bra{- \int \Phi(s\one_A(x)) \alpha (\dd x) } = \exp\bra{- \alpha(A) \Phi(s) } 
    = \bra{\frac{1 - p }{1-pe^{-s}}}^{\alpha(A)}, \quad s > 0
\end{align*}
which is the Laplace transform of a negative binomial distributed random variable with parameters $\alpha(A)$ and $p$. Moreover, \eqref{Eq: Laplace functional Pascal process} implies the independence of $\zeta(A_1), \ldots, \zeta(A_n)$ immediately.

\paragraph{Reversible Measure (Proof of Theorem~\ref{theorem: gSIP}(ii)).}
Let $Q_c = Q$ be the jump kernel from~\eqref{eq:sip-jumpkernel}. It is enough to check the detailed balance relation 
\begin{equation} \label{eq:sip-balance}
	\rho \otimes Q_c(\mathscr A \times \mathscr B) = \rho \otimes Q_c(\mathscr B\times \mathscr A)\quad \mathscr A, \mathscr B\in \mathcal N_{<\infty}, 
\end{equation}
where 
\[
	\rho \otimes Q_c(\mathscr A \times \mathscr B) = \int_\mathscr A \rho(\dd \eta) \iint \one_\mathscr B(\eta- \delta_x+ \delta_y) c(x,y) (\alpha + \eta)(\dd y) \eta(\dd x).
\] 
The proof idea is that for particularly simple choices of $c(x,y)$ and $\mathscr A,\mathscr B$, the relation~\eqref{eq:sip-balance} boils down to a detailed balance relation for a discrete inclusion process. 

We start with some preliminary observations. First, it is enough to prove \eqref{eq:sip-balance} for functions $c$ of the form 
\begin{equation} \label{eq:simple-c}
	c(x,y) = \sum_{i,j=1}^r d_{ij} \one_{C_i}(x) \one_{C_j}(y)
\end{equation}
with $r\in \N$, symmetric non-negative weights $d_{ij} = d_{ji} \geq 0$, and sets $A_1,\ldots, A_r\in \mathcal E$. Indeed, the set $\mathcal M$ of non-negative measurable functions $f:E\times E\to \R_+$ for which the symmetrized function $c(x,y):= \frac12 (f(x,y)+ f(y,x))$ satisfies~\eqref{eq:sip-balance} is closed under pointwise monotone limits. If~\eqref{eq:sip-balance} holds true for all functions $c$ of the form~\eqref{eq:simple-c}, then $\mathcal M$ contains all indicators $\one_{A\times B}$, $A,B\in \mathcal E$. The monotone class theorem then implies that $\mathcal M$ contains all bounded non-negative measurable functions. 

Second, by the $\pi$-$\lambda$ theorem, it is enough to check~\eqref{eq:sip-balance} for sets of the form
\begin{equation}
    \label{eq:simple-ab}
	\mathscr A = \bigcap_{j=1}^k \bigl\{\eta \in \mathbf N_{<\infty}:\, \eta(A_j) = m_j\bigr\}, \quad 
			\mathscr B = \bigcap_{j=1}^\ell \bigl\{\eta \in \mathbf N_{<\infty}:\, \eta(B_j) = n_j\bigr\}
\end{equation} 
with $k,\ell \in \N$, $A_i, B_j\in \mathcal E$, and $m_i,n_j\in \N_0$. 

Third, for the relation~\eqref{eq:sip-balance} to hold true for all $c(x,y)$ of the form~\eqref{eq:simple-c} and all sets $\mathscr A$, $\mathscr B$ of the form~\eqref{eq:simple-ab}, it is enough to consider the situation where $r = k = \ell$, $A_i = B_i = C_i$, and the sets $A_1,\ldots, A_r$ are pairwise disjoint, as the general case follows by taking linear combinations. 

In the situation of the last paragraph, we compute, for $\eta \in \mathscr A$, and assuming all diagonal elements $d_{ii}$ vanish,
\begin{align}
    \label{eq:sipq1}
	Q_c(\eta,\mathscr B) & = \sum_{i,j=1}^r d_{ij} \int_{A_i} \bra{ \int_{A_j} \one_\mathscr B(\eta - \delta_x+ \delta_y) (\alpha + \eta)(\dd y) } \eta(\dd x) \notag \\
	& = \sum_{i,j=1}^r d_{ij} \eta(A_i) \bigl(\alpha(A_j) + \eta(A_j) \bigr)
	\one_{\{\eta(A_i) -1 = n_i\}} \one_{\{\eta(A_j)+1 = n_j\}} \prod_{s\notin\{i,j\}} \one_{\{\eta(A_s) = n_s\}} \notag \\
	& = \sum_{\substack{1\leq i,j\leq r:\\ i \neq j}} d_{ij} m_i (\alpha(A_j) + m_j) \delta_{m_i-1,n_i} \delta_{m_j+1,n_j} \prod_{s\notin \{i,j\}} \delta_{m_s,n_s},
\end{align} 	
and we recognize the transition rates of the SIP  with state space $\N_0^r$. For non-zero diagonal elements, we need to add
\begin{equation} \label{eq:sipq2}
	\sum_{i=1}^r d_{ii} m_i (\alpha(A_i) + m_i) \prod_{s=1}^r \delta_{m_i,n_i}.
\end{equation} 
We denote the sum of~\eqref{eq:sipq1} and~\eqref{eq:sipq2} $q(m,n)$. Notice that for non-zero $d_{ii}$ we may have $q(m,m)>0$. 

Abbreviate $\alpha(A_j)=:\alpha_j$. For $j\in \{1,\ldots,r\}$ and $m_j\in \N_0$, set 
\[
	\pi_j(m_j):= \rho\bigl( \{\eta:\, \eta(A_j) = m_j\}\bigr) = (1-p)^{\alpha_j} \alpha_j(\alpha_j+1)\cdots (\alpha_j + m_j - 1) \frac{p^{m_j}}{m_j!}. 
\] 	
Further set $\pi(m) = \pi_1(m_1)\cdots \pi_r(m_r)$. Then 
\[ 
	\rho \otimes Q_c(\mathscr A,\mathscr B) = \pi(m) q(m,n). 
\]
A similar computation shows $\rho \otimes Q_c(\mathscr B,\mathscr A) = \pi(n) q(n,m)$. The symmetry relation~\eqref{eq:sip-balance} now reads $\pi(m) q(m,n)= \pi(n) q(n,m)$ which is the detailed balance relation
for the SIP. \qed

\subsection{Properties of Generalized Meixner Polynomials. Proof of Propositions~\ref{prop:in-to-meixner} and~\ref{prop:meixner-lambdan}} 

\label{sec:meixner-properties}

Let $p \in (0,1)$. Note that the generating function of monic Meixner polynomials, given by (see, e.g., \cite{HypergeometricOrthogonalPolynomials})
\begin{align*}
    e_t(x,a) := \sum_{n=0}^\infty \frac{t^n}{n!} \mathscr M_n(x;a; p) = \bra{\frac{1-p+t}{1-p+tp}}^x \bra{\frac{1-p}{1-p+tp}}^a, \quad t,a > 0, x \in \N_0,
\end{align*}
satisfies $e_t(x+y, a+b) = e_t(x, a) e_t(y,b)$ for each $t > 0$, $x,y \in \N_0$, $a,b > 0$. As a consequence, we get the convolution property (see, e.g., \cite{al1976convolutions})
\begin{align}
    \label{equation: Meixner convolution property}
    \mathscr M_n(x+y;a+b; p) = \sum_{k=0}^n \binom{n}{k} \mathscr M_k(x;a; p) \mathscr M_{n-k}(y;b; p).
\end{align}

\begin{proof}[Proof of Proposition~\ref{prop:in-to-meixner}]
	By the factorization property from Proposition~\ref{proposition: properties ort pol 2} it is enough to show 
	\begin{equation}
        \label{eq:ione-meixner}
		I_1(\one_{A^d},\eta) =\mathscr M_d(\eta(A);\alpha(A);p)
	\end{equation}
	for all $d\in \N$ and $A\in \mathcal E$. As we have chosen our univariate Meixner polynomials $\mathscr M_d$ to have leading coefficient one, we know that $\mathscr M_d(\eta(A);\alpha(A);p)$ is equal to $\eta(A)^d$ plus some polynomial in $\eta(A)$ of degree $\leq d-1$. Therefore \eqref{eq:ione-meixner} follows once we know that the map $\eta \mapsto \mathscr M_d(\eta(A);\alpha(A);p)$ is orthogonal to the space $\mathcal P_{d-1}$. We shall see that this identity follows from the convolution property~\eqref{equation: Meixner convolution property} and the complete independence.

	We check first that $\eta \mapsto \mathscr M_d(\eta(A);\alpha(A);p)$  is orthogonal in $L^2(\rho)$ to all maps $\eta\mapsto \eta^{\otimes m}(C)$, for every $m\leq d-1$ and $C\in \mathcal E^{m}$ with $C\subset A^m$. When $C=A^m$, we are looking at two univariate polynomials in the variable $x=\eta(A)$ and the orthogonality relation follows from the orthogonality of the univariate Meixner polynomials $x\mapsto \mathscr M_d(\eta(A);\alpha(A);p)$ to the monomial $x\mapsto x^m$. The orthogonality to constant functions ($m=0$) follows from univariate orthogonality as well. 
	
	Next consider the case $C= C_1^{d_1}\times\cdots \times C_N^{d_N}$ with $N\in \N$, $d_1,\ldots,d_N\in\N$ with $d_1+\cdots + d_N \leq d-1$ and  pairwise disjoint measurable sets $C_i\subset A$. Suppose first that $C_1\cup \cdots \cup C_N=A$. We use the convolution property~\eqref{equation: Meixner convolution property} and the complete independence of the Pascal point process to find
	\begin{multline} \label{eq:mortho2} 
		\int \mathscr M_d(\eta(A);\alpha(A);p) \eta^{\otimes m}(C) \rho(\dd \eta) \\
			= \sum_{k_1+\cdots + k_N=m} \binom{m}{k_1,\ldots, k_{N}} \prod_{i=1}^N \int\mathscr M_{k_i}\bigl( \eta(C_i);\alpha(C_i);p\bigr) \eta^{\otimes d_i}(C_i) \rho(\dd \eta). 
	\end{multline} 
	In each summand, we must have $d_i< k_i$ for at least one $i\in \{1,\ldots,N\}$ and therefore by the orthogonality of univariate Meixner polynomials, at least one of the integrals on the right side above vanish. As a consequence, 
	\begin{equation}\label{eq:mortho}
			\int \mathscr M_d(\eta(A);\alpha(A);p) \eta^{\otimes m}(C) \rho(\dd \eta) =0 
	\end{equation} 	
	This holds true as well when each $C_i$ is contained in $A$ and $C_{N+1}:= A\setminus (C_1\cup\cdots \cup C_N)$ is non-empty. In that case we use a similar decomposition but now the sum on the right side of~\eqref{eq:mortho2} is over $(k_1,\ldots, k_{N+1})$ and the product has an additional factor $\int \mathscr M_{k_{N+1}}(\eta(C_{N+1});\alpha(C_{N+1});p)\rho(\dd \eta)$. 
	
	Every Cartesian product $C= D_1\times \cdots \times D_m$ contained in $A^m$ is a disjoint union of finitely many Cartesian products in which any two distinct factors are either distinct or equal. Therefore, by linearity, the orthogonality relation~\eqref{eq:mortho} extends to all such sets. 
	The functional monotone class theorem yields the orthogonality of the generalized Meixner polynomial to all maps of the form $\eta\mapsto \eta^{\otimes m}(f_m)$ with bounded measurable $f_m:E^m\to \R$ supported in $A^m$ and then, by linearity, the orthogonality to all linear combinations of such maps.
	
	In the notation of Lemma~\ref{lem:localize} below, we have checked the orthogonality of $\mathscr M_{d}(\eta(A);\alpha(A);p)$ to $\mathcal P_{d-1}(A)$. Using complete independence and arguments similar to the proof of Lemma~\ref{lem:localize}, we conclude that the Meixner polynomial is in fact orthogonal to $\mathcal P_{d-1}$. This completes the proof of the proposition. 
\end{proof}

\begin{proof}[Proof of Proposition~\ref{prop:meixner-lambdan}]
	The orthogonality of $I_n(f_n,\sd)$ and $I_m(g_m,\sd)$ for $m\neq n$ is an immediate consequence of  the definition of generalized orthogonal polynomials, it does not use any properties of the Pascal distribution $\rho$. Thus we need only treat the case $m=n$.
	
	Using linearity and the monotone class theorem as in the proof of Proposition~\ref{prop:in-to-meixner}, one finds that it suffices to show  the orthogonality relation for functions $\tilde f_n$, $\tilde g_n$ that are symmetrized versions of indicator functions $f_n,g_n:E^n\to \R$ of the form
	\[
		f_n = \one_{B_1}^{\otimes d_1} \otimes \cdots \otimes \one_{B_N}^{\otimes d_N},\quad g_n = \one_{B_1}^{\otimes d'_1} \otimes \cdots \otimes \one_{B_N}^{\otimes d'_N}
	\]
	with $B_1, \ldots, B_N \in \mathcal{E}$ disjoint, and $\sum_{i=1}^N d_i = \sum_{i=1}^N d'_i = n$. 
	The symmetrization of $f_n$ is given by 
	\[
		\tilde f_n(x_1,\ldots, x_n) = \frac{1}{n!} \sum_{\sigma\in \mathfrak S_n} f_n(x_{\sigma(1)},\ldots, x_{\sigma(n)}),
	\] 
	the symmetrization $\tilde g_n$ of $g_n$ is defined in a similar way. Notice that $I_n(\tilde f_n,\eta) = I_n(f_n,\eta)$ and $I_n(\tilde g_n,\eta) = I_n(g_n,\eta)$ but in general $\int \tilde f_n \tilde g_n\dd \lambda_n \neq\int f_n g_n\dd\lambda_n$.

	Proposition~\ref{prop:in-to-meixner}, the complete independence, and the orthogonality relation~\eqref{eq:mortho-univariate}  for univariate Meixner polynomials yield
	\begin{equation} \label{eq:inhnkn}
		\int I_n(\tilde f_n,\eta)\, I_n(\tilde g_n,\eta)\, \rho(\dd \eta) 
		= \prod_{i=1}^N \one_{\{d_i = d'_i\}} \frac{d_i!p^{d_i}}{(1-p)^{2d_i}}\, (\alpha(B_i))^{(d_i)}.
	\end{equation}
	If $d_i \neq d'_i$ for at least one $i$, then the right side is zero, moreover $\tilde f_n\, \tilde g_n$ vanishes identically. Hence in that case 
	\[
			\int I_n(\tilde f_n,\eta)\, I_n(\tilde g_n,\eta)\, \rho(\dd \eta) 
 			 = 0 = \int \tilde f_n \tilde g_n \dd \lambda_n.
	\]	
	and the required equality holds true. 
	
	If $d_i = d'_i$ for all $i$, then $f_n= g_n$ on $E^n$. By the definition of $\lambda_n$, we have 
	\[
		\int f_n^2\, \dd \lambda_n = \lambda_n( B_1^{d_1}\times \cdots \times B_n^{d_n}) = \prod_{i=1}^N (\alpha(B_i))^{(d_i)}
	\] 
	hence \eqref{eq:inhnkn} gives 
	\begin{equation} \label{eq:hniso}
		\int \bigl(I_n(\tilde f_n,\eta)\bigr)^2\, \rho(\dd \eta) = \Bigl( \prod_{i=1}^N d_i!\Bigr) \int f_n^2 \dd \lambda_n. 	
	\end{equation}
	Next we check that the product of factorials on the right side disappears when $f_n$ is replaced by the  symmetrized function $\tilde f_n$. For $\sigma \in \mathfrak S_n$ and $x=(x_1,\ldots,x_n) \in E^n$, let $x_\sigma := (x_{\sigma(1)},\ldots, x_{\sigma(n)})$. Then, by the permutation invariance of the measure $\lambda_n$, we have
	\[
		\int \tilde f_n^2\, \dd \lambda_n 		= \frac{1}{n!^2}\sum_{\sigma,\tau\in \mathfrak S_n} \int f_n(x_\sigma) f_n(x_\tau) \lambda_n(\dd x)  = \frac{1}{n!} \sum_{\pi \in \mathfrak S_n} \int f_n(x_\pi) f_n(x) \lambda_n(\dd x). 
	\]
	Because of the disjointness of the sets $B_i$, the product $f_n(x_\pi)f_n(x)$ vanishes unless $\pi$ leaves the sets $\{1,\ldots, d_1\}$, $\{d_1+1,\ldots, d_1+d_2-1\}$ etc.\ invariant, and in the latter case $f_n(x_\pi)f_n(x) = f_n(x)^2$. The number of relevant permutations is equal to $d_1!\cdots d_N!$. As a consequence,
	\[
		\int{\tilde f_n}^2\, \dd \lambda_n = \frac{1}{n!}\Bigl(\prod_{i=1}^N d_i!\Bigr) \int f_n^2 \dd \lambda_n. 
	\]	
	By~\eqref{eq:hniso}, we get
	\[
		\int \bigl( I_n(\tilde f_n,\eta)\bigr)^2 \rho(\dd \eta) = \frac{n! p^n}{(1-p)^{2n}} \int \tilde f_n^2\, \dd \lambda_n
	\] 
    which is the required equality (remember $\tilde f_n= \tilde g_n$). 
\end{proof}

\begin{appendices}

\section{Properties of Generalized Orthogonal Polynomials}
\label{appendix: properties ort pol}

This section is devoted to the proof of Propositions~\ref{proposition: properties ort pol 1} and~\ref{proposition: properties ort pol 2}.  Section \ref{section: ortog falling factorial pol} deals with the proofs of \eqref{equation: spanned by Jk} and \eqref{equation: projection Jk}, while Section \ref{section: appendix proof of factorization} deals with the proof of the factorization property of the generalized orthogonal polynomials given in \eqref{equation: factorization of orthogonal polynomials}.

\subsection{Orthogonalization of Generalized Falling Factorial Polynomials}
\label{section: ortog falling factorial pol}

Proposition~\ref{proposition: properties ort pol 1} follows from explicit formulas that link factorial measures $\eta^{(n)}$ and product measure $\eta^{\otimes n}$. These relations are similar to  relations between moments and factorial moments of integer-valued random variables with Stirling numbers, see \cite[Chapter 5]{daley-vere-jones-vol1}. A systematic treatment in terms of Stirling operators is found in \cite{finkelshtein2021stirling}.

\begin{proof}[Proof of \eqref{equation: spanned by Jk}]
In order to show that $\mathcal P_n$ is the linear hull of generalized falling factorials $J_k(f_k,\eta)$, $k\leq n$, it is enough to check that every monomial $\eta\mapsto \eta^{\otimes n}(f)$ is a linear combination of falling factorials of degree $k\leq n$ and vice-versa. 

Let $\eta = \delta_{x_1}+\cdots + \delta_{x_\kappa} \in \mathbf N_{<\infty}$ and $f:E^n\to \R$ a bounded measurable function. Then 
\[
	\eta^{\otimes n} (f_n) = \sum_{1\leq i_1,\ldots, i_n\leq \kappa} f_n(x_{i_1},\ldots, x_{i_n}). 
\] 
Every multi-index $(i_1,\ldots, i_n)$ on the right side gives rise to a set partition $\sigma$ of $\{1,\ldots, n\}$ in which $k$ and $\ell$ belong to the same block if and only if $i_k=i_\ell$. 
Denote by $\Sigma_n$ the set of partitions of $\set{1, \ldots, n}$. For $\sigma\in \Sigma_n$, let $|\sigma|$ be the number of blocks  of the set partition. Further let $(f_n)_\sigma : E^{\abs{\sigma}} \to \R$ be the function obtained from $f_n$ by identifying, in order of occurrence, those arguments which belong to the same block of $\sigma$. 
Grouping multi-indices $(i_1,\ldots,i_n)$ that give rise to the same partition $\sigma$, we find
\[
	\int f_n \dd \eta^{\otimes n} = \sum_{\sigma\in \Sigma_n} \int (f_n)_{\sigma}\, \dd \eta^{(|\sigma|)}
\] 
(compare \cite[Exercise 5.4.5]{daley-vere-jones-vol1}) and conclude that $\eta^{\otimes n}(f_n)$ is a linear combination of generalized falling factorials of degrees $|\sigma|\leq n$. 

Conversely, 
\begin{equation} \label{equation: factorial measure as sum of product measures}
	\int f_n \dd \eta^{(n)} = \sum_{\sigma\in \Sigma_n} (-1)^{n-|\sigma|} \int (f_n)_{\sigma}\, \dd \eta^{\otimes |\sigma|}
\end{equation} 
hence the falling factorial of degree $n$ on the left side is a linear combination of monomials $\eta^{\otimes k}(g_k)$ of degree $k\leq n$. 
\end{proof}

\begin{proof}[Proof of \eqref{equation: projection Jk}]
For $n = 0$ the identities $J_0(f_0, \eta) = f_0 = \eta^{\otimes 0}(f_0)$ yield \eqref{equation: projection Jk}. For $n \in \N$,  we notice that \eqref{equation: factorial measure as sum of product measures} implies 
\[
	\int f_n\dd \eta^{(n)} = \int f_n\dd \eta^{\otimes n} + Q(\eta)
\] 
for some $Q\in \mathcal P_{n-1}$, given by a sum over set partitions with a number of blocks $|\sigma|\leq n-1$. It follows that $\eta\mapsto J_n(f_n,\eta)$ and $\eta\mapsto \eta^{\otimes n}(f_n)$ have the same orthogonal projections onto $(\mathcal P_{n-1})^{\perp}$. 
\end{proof}

\subsection{Factorization Property of Generalized Orthogonal Polynomials}

\label{section: appendix proof of factorization}

In order to exploit the complete independence, it is helpful to check that  if $f:E^n\to \R$ is supported in $A^n$, then $I_n(f,\eta)$ depends only on what happens inside $A$. We show a bit more. Let $\mathcal P_n(A)\subset \mathcal P_n$ be the space of linear combinations of maps $\eta\mapsto \eta^{\otimes k}(f_k)$, $k\leq n$, with bounded measurable $f_k:E\to \R$ vanishing on $E^k\setminus A^k$. Notice that every function $F\in \mathcal P_n(A)$ depends only on the restriction $\eta_A$, defined by $\eta_A(B):= \eta(A\cap B)$. 

\begin{lemma}\label{lem:localize}
	Let $d\in \N$, $A\in \mathcal E$, and $f:E^d\to \R$ a bounded measurable function that vanishes outside $E^d\setminus A^d$. Then there exists a map $Q\in \overline{\mathcal P_{d-1}(A)}$ such that $I_d(f,\eta)= \eta^{\otimes d}(f) - Q(\eta)$ for $\rho$-almost all $\eta\in \mathbf N_{<\infty}$. 
\end{lemma}

\begin{proof} 
	Let $Q$ be the orthogonal projection of $\eta \mapsto \eta^{\otimes d}(f)$ onto $\overline{\mathcal P_{d-1}(A)}$. Then $Q\in \overline{\mathcal P_{d-1}(A)}$ and the difference $F(\eta):= \eta^{\otimes d}(f)- Q(\eta)$
	is orthogonal to $\overline{\mathcal P_{d-1}(A)}$. We exploit the complete independence to show that  $F$ is actually orthogonal to the bigger space $\overline{\mathcal P_{d-1}}$. 
	
	Let $n\in \{1,\ldots, d-1\}$. If $C\in \mathcal E^n$ is of the form $C_1\times C_2$ with $C_i\in \mathcal E^{s_i}$ where $s_1,s_2\in \N_0$ and  $C_1\subset A$, $C_2\subset A^c$, then $\eta^{\otimes n}(C) = \eta^{\otimes s_1}(C_1)\eta^{\otimes s_2}(C_2)$ and by the complete independence (notice $F(\eta) = F(\eta_A)$) 
	\[
		\int F(\eta) \eta^{\otimes n}(C) \rho(\dd \eta) = \Bigl(\int F(\eta) \eta^{\otimes s_1}(C_1) \rho(\dd \eta)\Bigr) \Bigl(\int \eta^{\otimes s_2}(C_2)\rho(\dd \eta)\Bigr). 
	\] 
	The first integral on the right side vanishes because of $C_1\subset A^{s_1}$, $s_1\leq d-1$, and $F\perp \overline{\mathcal P_{d-1}(A)}$. Therefore $F$ is orthogonal to $\eta\mapsto \eta^{\otimes n}(C)$. 
	
	More generally, every set $C\in \mathcal E^n$ is the disjoint union of Cartesian products $C_1\times \cdots \times C_n$ in which every $C_i$ is either contained in $A$ or in $A^c$. Taking linear combinations and exploiting that $\eta^{\otimes n}(g)$ does not change if we permute variables in $g$, we find that $F$ is orthogonal to $\eta^{\otimes n}(C)$ for all $C\in \mathcal E^n$ and then, by the usual measure-theoretic arguments, to all maps $\eta\mapsto \eta^{\otimes n}(g)$, $g:E^n\to \R$ bounded and measurable. 
	The map $F$ is also orthogonal to all constant functions because every constant function is in $\overline{\mathcal P_{d-1}(A)}$. 
	
	Hence, taking linear combinations of maps $\eta^{\otimes n}(f_n)$, $n\in \{0,\ldots, d-1\}$, we see that $F$ is orthogonal to the space $\overline{\mathcal P_{d-1}}$. As $F(\eta) = \eta^{\otimes n}(f) - Q(\eta)$ with $Q\in \overline{\mathcal P_{d-1}}$, it follows that $I_n(f,\eta)= F(\eta)$ for $\rho$-almost all $\eta$.  
\end{proof} 

When evaluating the product of two generalized orthgonal polynomials $I_n(f,\eta)$ using Lemma~\ref{lem:localize}, it is important to know that the product of two polynomials is again a polynomial. 

\begin{lemma}\label{lem:product-polynomials}
	Let $A$ and $B$ be two disjoint measurable subsets of $E$ and $m,n\in \N_0$. Pick $F\in \overline{\mathcal P_m(A)}$ and $G\in \overline{\mathcal P_n(B)}$. Then $FG$ is in $\overline{\mathcal P_{m+n}(A\cup B)}$. 
\end{lemma}

\begin{proof}
	Write $||\cdot||$ for the $L^2(\rho)$-norm. 
	Let $(F_k)_{k\in \N}$ and $(G_k)_{k\in \N}$ be sequences in $\mathcal P_m(A)$ and $\mathcal P_n(B)$, respectively, with $||F-F_k||\to 0$ and $||G-G_k||\to 0$. We have $F_k(\eta) = F_k(\eta_A)$ for all $k$ and $\eta$ hence $F(\eta) = F(\eta_A)$ for $\rho$-almost all $\eta$. Similarly $G_k$ and $G$ depend on $\eta_B$ only. The triangle inequality and the complete independence yield
	\begin{align*} 
		||FG - F_kG_k|| &\leq || (F-F_k) G|| + || F_k(G- G_k)|| \\
			& = || F- F_k||\, ||G|| + ||F_k||\, ||G-G_k|| \to 0.
	\end{align*} 
	As each product $F_kG_k$ is in $\mathcal P_{m+n}(A\cup B)$, the limit $FG$ is in the closure $\overline{\mathcal P_{m+n}(A\cup B)}$.
\end{proof} 

\begin{proof}[Proof of Proposition~\ref{proposition: properties ort pol 2}]
    It is enough to treat the case $N=2$; the general case follows by an induction over $N$. Let $A_1$ and $A_2$ be two disjoint measurable subsets in $\mathcal E$. Let $d_1,d_2$ be two integers and $f_i:E^{d_i}\to \R$, $i=1,2$ be two bounded measurable functions that vanish outside $A_1^{d_1}$ and $A_2^{d_2}$ respectively.  
    By Lemma~\ref{lem:localize}, there exist maps $Q_i\in \overline{\mathcal P_{d_i-1}(A_i)}$, $i=1,2$, such that 
    \[
    	I_{d_1}(f_1,\eta) = \eta^{\otimes d_1} (f_1) - Q_1(\eta),\quad I_{d_2}(f_2,\eta)= \eta^{\otimes d_2} (f_2) - Q_2(\eta)
    \]
    for $\rho$-almost all $\eta$. Therefore by Lemma~\ref{lem:product-polynomials}, we have 
    \begin{equation} \label{eq:iproduct}
    	I_{d_1}(f_1,\eta) I_{d_2}(f_2,\eta) = \eta^{\otimes d_1}(f_1)\eta^{\otimes d_2} (f_2) - Q(\eta)
    \end{equation}
    with $Q\in \overline{\mathcal P}_{d_1+d_2-1}$. 
    Let $s_1,s_2\in \N_0$ and $C_1\in \mathcal E^{s_1}$, $C_2\in \mathcal E^{s_2}$ with $s_1+s_2\leq d_1+d_2-1$ and $C_1\subset A_1^{s_1}$, $C_2\subset (A_1^c)^{s_2}$. Then, by the complete independence, 
    \[
    	\int I_{d_1}(f_1,\eta) I_{d_2}(f_2,\eta)\, \eta^{\otimes(s_1+s_2)}(C_1\times C_2)\rho(\dd\eta) 
    	 = \prod_{i=1}^2 \int I_{d_i}(f_i,\eta)\, \eta^{\otimes s_i}(C_i) \rho(\dd \eta). 
    \] 
    We must have $s_1\leq d_1-1$ or $s_2\leq d_2-1$, therefore at least one of the integrals on the right side vanishes and the product $I_{d_1}(f_1,\eta) I_{d_2}(f_2,\eta)$ is orthogonal to $\eta^{\otimes n}(C)$. We conclude with an argument similar to the proof of Lemma~\ref{lem:localize} that $I_{d_1}(f_1,\eta) I_{d_2}(f_2,\eta)$ is in fact orthogonal to $\overline{\mathcal P}_{d_1+d_2-1}$. It follows that the product is equal to $I_{d_1+d_2}(f_1\otimes f_2,\eta)$ for $\rho$-almost all $\eta$. 
\end{proof}

\end{appendices}

\bibliographystyle{acm}
\bibliography{main}

\end{document}